\documentclass[reqno,english]{amsart}
\usepackage[utf8]{inputenc}
\usepackage{amsmath,amssymb,amsthm,mathrsfs,color,times,textcomp,yfonts,mathtools,cases}

\allowdisplaybreaks[4]

\usepackage{bm}

\usepackage[T1]{fontenc} 

\usepackage{subcaption}
\usepackage[normalem]{ulem}
\usepackage[export]{adjustbox}
\usepackage{esint}
\usepackage{xcolor}
\usepackage{array}
\usepackage[colorlinks=true]{hyperref}
\hypersetup{urlcolor=blue, citecolor=blue, linkcolor=red}

\hypersetup{
colorlinks=true,
linkcolor=red
}
\usepackage{indentfirst}
\usepackage{graphicx}
\usepackage{float}
\numberwithin{equation}{section}
\newtheorem{theorem}{Theorem}[section]
\newtheorem{lemma}[theorem]{Lemma}

\newtheorem{proposition}[theorem]{Proposition}

\newtheorem{definition}[theorem]{Definition}
\newtheorem{conjecture}{Conjecture}

\newtheorem{othertheorem}{Theorem}

\newtheorem{remark}{Remark}[section]

\newcommand{\Rom}[1]{%
\textup{\uppercase\expandafter{\romannumeral#1}}%
}
\newcommand\norm[1]{\left\lVert#1\right\rVert}

\newcommand*{\ud}{d}
\newcommand{\R}{\mathbb R}

\usepackage{mathtools}
\usepackage{thmtools}
\declaretheoremstyle[headfont=\normalfont]{normalhead}
\usepackage{color}
\usepackage{morefloats}
\usepackage{refcheck}
\usepackage{textcomp}
\usepackage{url}

\title[Generalized Chang-Yang conjecture]{A positive answer to the generalized Chang-Yang conjecture on $\mathbb{S}^N$}

\author{Changfeng Gui}
\address{Department of Mathematics, University of Macau, Taipa, Macau}
\email{changfenggui@um.edu.mo}

\author{Tuoxin Li}
\address{Department of Mathematics \\  The Chinese University of Hong Kong \\ Shatin \\ NT \\ Hong Kong}
\email{txli@math.cuhk.edu.hk}

\author{Juncheng Wei}
\address{Department of Mathematics \\  The Chinese University of Hong Kong \\ Shatin \\ NT \\ Hong Kong}
\email{wei@math.cuhk.edu.hk}

\author{Zikai Ye}
\address{Department of Mathematics \\  The Chinese University of Hong Kong \\ Shatin \\ NT \\ Hong Kong}
\email{zkye@math.cuhk.edu.hk}
\begin{document}



\begin{abstract}
We prove that for every integer $N\geq 3$ and $\alpha\geq \frac{1}{2}$, Beckner's inequality
\begin{equation*}
\frac{\alpha}{2}\int_{\mathbb{S}^N}u(P_{N}u)
dw+(N-1)!\int_{\mathbb{S}^N}u
dw-\frac{(N-1)!}{N}\log\int_{\mathbb{S}^N}e^{Nu}
dw\geq 0
\end{equation*}
holds for every $u\in H^{\frac{N}{2}}(\mathbb{S}^N)$ whose center of mass is at the origin. The proof is mainly based on an integral representation formula and a rigidity theorem for stable critical points. Hence, we answer the generalized Chang-Yang conjecture positively for every integer $N\geq 3$. 
\end{abstract}

\maketitle
\hypersetup{linkcolor=black}
\tableofcontents

\section{Introduction and main results}

\subsection{Sharp Beckner's inequality}

For $\alpha>0$, we consider the functional
\begin{equation*}
J_{\alpha,N}(u):=\frac{\alpha}{2}\int_{\mathbb{S}^N}u(P_{N}u)
dw+(N-1)!\int_{\mathbb{S}^N}u
dw-\frac{(N-1)!}{N}\log\int_{\mathbb{S}^N}e^{Nu}
dw
\end{equation*}
for $u\in H^{\frac{N}{2}}(\mathbb{S}^N)$, where $dw$ denotes the normalized surface measure on $\mathbb{S}^N$ so that $\int_{\mathbb{S}^N}dw=1$. Here
\begin{equation*}
\begin{aligned}
    P_N:=
    \begin{cases}
        \prod_{k=0}^{\frac{N-2}{2}}(-\Delta+k(N-k-1)),&\text{ for }N\text{ even}, \\
        \left(-\Delta+(\frac{N-1}{2})^2\right)^{\frac{1}{2}}\prod_{k=0}^{\frac{N-3}{2}}(-\Delta+k(N-k-1)),&\text{ for }N\text{ odd},
    \end{cases}
\end{aligned}
\end{equation*}
where the square root is understood in the spectral sense. More precisely, let $\mathcal{H}_k$ be the space of spherical harmonics of degree $k$ on $\mathbb{S}^N$, then for every spherical harmonic $Y_k\in \mathcal{H}_k$ with $k\geq 1$,
\begin{equation}\label{PN:eigenvalue}
P_NY_k=\frac{\Gamma(k+N)}{\Gamma(k)}Y_k,
\end{equation}
and $P_N1=0$.

In his seminal paper \cite{Beckner1993}, Beckner proved the higher-order Moser-Trudinger-type inequality
\begin{equation*}
J_{1,N}(u)\geq 0,    \qquad u\in H^{\frac{N}{2}}(\mathbb{S}^N).
\end{equation*}
This inequality is referred to as Beckner's inequality.

When $N=4$, $P_4$ represents the Paneitz operator on $\mathbb{S}^4$ introduced by Paneitz in \cite{Paneitz2008}. When $N$ is even, the operator $P_N$ belongs to the family of conformally invariant operators $P_{N,g,k}$ with $k=\frac{N}{2}$, constructed by Graham, Jenne, Mason, and Sparling on a Riemannian manifold $(M^N, g)$ in \cite{GJMS1992}, called the GJMS operator. When $N$ is odd, $P_N$ is a nonlocal elliptic pseudo-differential operator of order $N$ and the square root is understood in the spectral sense. In both cases, for $k=\frac{N}{2}$, on $(\mathbb{S}^N,g)$, under conformal change $\tilde g=e^{2\omega}g$, the covariance law
\begin{equation}\label{conformal PN g}
P_{N,\tilde g}(\phi)=e^{-N\omega}P_{N,g}(\phi)
\end{equation}
is satisfied. We refer to \cite{CLY2019, CM2023, ChangYang1995, ChangYang1997, DHL2000,DM2008, FG2013, GHX2021, GurMal2015, LiXiong2019, Mal2006, WX1998} and the references therein for results and background on $Q$-curvature problems and GJMS operators. 

In addition, let $\xi=(\xi_1,\dots,\xi_{N+1})\in \mathbb{S}^N$,
A. Chang and P. Yang \cite{ChangYang1995} and Wei-Xu \cite{WX1998} proved that, if $u$ belongs to the zero center-of-mass set
\begin{equation*}
\mathcal{L}_{N}=\left\{u\in H^{\frac{N}{2}}(\mathbb{S}^N)\ :\ \int_{\mathbb{S}^N}e^{Nu} \xi_j dw=0,\ j=1,\cdots, N+1 \right \},
\end{equation*}
then for any $\alpha>\frac{1}{2}$, there exists a constant $C(\alpha,N)\geq0$ such that 
\begin{equation*}
J_{\alpha,N}(u)\geq -C(\alpha,N)    
\end{equation*}
for any $u\in \mathcal{L}_{N}$. This lower bound motivates the generalized Chang-Yang conjecture that, as in dimension two, $C(\alpha,N)$ can be chosen to be zero.

\begin{conjecture}[Generalized Chang-Yang conjecture]\label{Chang-Yang conj}
For any $N\geq 1$ and $\alpha \geq \frac{1}{2}$, 
\begin{equation*}
    \inf_{u \in \mathcal{L}_{N}} J_{\alpha,N}(u)=0.
\end{equation*}
\end{conjecture}

A critical point of $J_{\alpha,N}$ under the zero center of mass constraint satisfies the following $Q$-curvature-type equation on $\mathbb{S}^N$:
\begin{equation*}
\alpha P_N u+(N-1)!(1-\frac{e^{Nu}}{\int_{\mathbb{S}^N}e^{Nu}dw})=\sum_{i=1}^{N+1}a_i \xi_i e^{Nu} \ \mbox{on} \ \mathbb{S}^N
\end{equation*}
for some constants $a_i$, $i=1,\dots, N+1$.

Due to conformal invariance of $J_{1,N}$, A. Chang and P. Yang \cite{ChangYang1995} and Wei-Xu \cite[Step 1 in Theorem 2.6]{WX1998} showed that the following Kazdan-Warner condition 
\begin{equation*}
\int_{\mathbb S^N}\langle \nabla Q, \nabla \xi_i\rangle e^{Nu} \ud w=0,\ i=1,\dots, N+1,
\end{equation*}
holds for the prescribed $Q$-curvature equation
\begin{equation*}
P_{N}u + (N-1)!- Qe^{Nu}=0 \text{ on }\mathbb S^N.
\end{equation*}
It follows that for $0< \alpha\leq 1$, all Lagrange multipliers $a_i$ vanish. Thus we have the following result:
\begin{proposition}\label{Lagrange}
For every $0<\alpha\leq 1$, if $u$ is a critical point of $J_{\alpha,N}$ with constraint $\mathcal{L}_{N}$, then $u$ is a solution to
\begin{equation}\label{paneitz}
\alpha P_N u+(N-1)!\left(1-\frac{e^{Nu}}{\int_{\mathbb{S}^N}e^{Nu}dw}\right)=0 \ \mbox{on} \ \mathbb{S}^N.
\end{equation}
\end{proposition}

For $\frac{1}{2}\leq \alpha<1$, if \eqref{paneitz} admits only constant solutions, then Conjecture \ref{Chang-Yang conj} is valid. When $\alpha<1$ is sufficiently close to $1$, Wei and Xu \cite{WX1998} proved that all solutions to \eqref{paneitz} are constants, but the full range $\alpha\in[\frac{1}{2},1)$ remains open.

On $\mathbb{S}^2$, the original conjecture of A. Chang and P. Yang in \cite{ChangYang1987, ChangYang1988} is that
\begin{conjecture}[Chang-Yang conjecture]
For $\alpha \geq \frac{1}{2}$, 
\begin{equation*}
    \inf_{u \in \mathcal{L}_{2}} J_{\alpha,2}(u)=0.
\end{equation*}
\end{conjecture}

The corresponding problem on $\mathbb{S}^2$ is known as the Nirenberg problem:
\begin{equation*}
-\alpha \Delta u + 1-  \frac{e^{2u}}{\int_{\mathbb{S}^2} e^{2u}}=0 \ \ \mbox{on} \ \mathbb{S}^2.
\end{equation*}
This problem has been extensively studied over the past four decades. See \cite{ChangYang1987, ChangYang1988, JinLiXiong2017} and the references therein. Feldman, Froese, Ghoussoub, and the first author  \cite{FFGG1998} first established the conjecture for axially symmetric functions when $\alpha >\frac{16}{25}-\epsilon$. Here we say a function $u$ on $\mathbb{S}^2$ is axially symmetric if, up to a rotation, $u=u(x)$ for $x=x_{1}\in (-1,1)$. The first and the third authors \cite{GW2000} subsequently proved the sharp axially symmetric version of the conjecture.  Later, Ghoussoub and Lin \cite{GL2010} showed that all critical points of $J_{\alpha,2}$ are axially symmetric and hence the conjecture holds true for $\frac{2}{3}-\epsilon<\alpha<1$. Finally, the first author and Moradifam \cite{GM2018} established the sphere covering inequality (see also \cite{GHM2020}) to prove that all solutions are axially symmetric. With the help of the first and the third authors' result \cite{GW2000} for axially symmetric case, they showed the full conjecture.
\begin{othertheorem}[\cite{GM2018}]\label{GM2018}
Chang-Yang conjecture holds.    
\end{othertheorem}

Furthermore, Shi, Sun, Tian and Wei \cite{SSTW2019} showed that all even solutions are axially symmetric for $\frac{1}{4}\leq \alpha<1$. For more general results on improved Moser-Trudinger-Onofri inequalities on $\mathbb{S}^2$ and their connections with the Szeg\"o limit theorem, see \cite{ChangGui202, ChangHang2022}.

For the higher-dimensional equation \eqref{paneitz} on $\mathbb{S}^N$, as mentioned above, the third author and Xu proved the conjecture when $\alpha$ is close to $1$. Since direct treatment of Conjecture \ref{Chang-Yang conj} seems difficult, in view of the work of Ghoussoub-Lin \cite{GL2010} and  the work of the first author and Moradifam \cite{GM2018}, the conjecture under axial symmetry is a natural and crucial first step in higher dimensions. Several results have been obtained for axially symmetric solutions with $N=4,6,8$. Gui-Hu-Xie \cite{GHW2022} obtained non-constant solutions by bifurcation methods for $\frac{1}{N+1}<\alpha<\frac{1}{2}$. They also proved the axially symmetric version of the conjecture for $\alpha\geq 0.517$ ($N=4$), $\alpha \geq 0.6168$ ($N=6$) and $ \alpha \geq 0.8261$ ($N=8$). The sharp bound $\alpha\geq \frac{1}{2}$ is obtained by Li-Wei-Ye \cite{LWY2022} ($N=4$) and Gui-Li-Wei-Ye \cite{GLWY2025} using refined estimates on Gegenbauer polynomials. 

For odd $N$, there are few results on the sharp Beckner's inequality.  When $N=1$, Beckner's inequality becomes the classical Lebedev-Milin inequality.

\begin{othertheorem}[Lebedev-Milin inequality]
    For any $u\in H^1(\mathbb D)$ with $\int_{\mathbb S^1}u\ud \theta=0$, where $\mathbb D$ is the unit disk in $\mathbb R^2$,
    \begin{equation*}
        \log\left(\frac{1}{2\pi}\int_{\mathbb S^1}e^u \ud \theta\right)\leq \frac{1}{4\pi}\norm{\nabla u}_{L^2(\mathbb D)}^{2}.
    \end{equation*}
\end{othertheorem}

Using Szeg\"o limit theorem, Widom \cite{Widom1988} improved the best constant under extra orthogonality conditions.
\begin{othertheorem}\label{szego}
    For any $u\in H^1(\mathbb D)$ with $\int_{\mathbb S^1}u\ud \theta=0$ and $\int_{\mathbb S^1}e^{u}e^{ik\theta}\ud \theta=0$, for $k=1,\dots,m$,
    \begin{equation*}
        \log\left(\frac{1}{2\pi}\int_{\mathbb S^1}e^u \ud \theta\right)\leq \frac{1}{4(m+1)\pi}\norm{\nabla u}_{L^2(\mathbb{D})}^{2}.
    \end{equation*}
    Equivalently, for $ \alpha \geq \frac{1}{m+1}$,
    \begin{equation*}
        \frac{\alpha }{2}\int_{\mathbb S^1}(P_1 u)u\ud w+\int_{\mathbb S^1}u\ud w-\log\int_{\mathbb S^1}e^{u}\ud w \geq 0.
    \end{equation*}
\end{othertheorem}

For odd $N\geq 3$, the operator $P_N$ is a nonlocal elliptic pseudo-differential operator of order $N$. Wei-Xu \cite{WX1998} proved the generalized Chang-Yang conjecture when $\alpha$ is close to $1$. Zhang \cite{Zhang2025} recently proved the rigidity for $\alpha>1$. To the best of our knowledge, there are no other results. The non-locality of $P_N$ has been one of the main obstruction in the study of the odd-dimensional case using the integration-by-parts methods in previous even-dimensional papers \cite{GHX2021, GHW2022, GLWY2025, LWY2022}.

In this paper, we prove the generalized Chang-Yang conjecture in every dimension $N\geq 3$.
\begin{theorem}\label{main}
Let $N\geq 3$ be an integer and $\alpha\geq \frac{1}{2}$. Then
\begin{equation*}
\inf_{ u\in {\mathcal{L}_{N}}} J_{\alpha,N} (u)=0.    
\end{equation*}
Moreover, for $u\in \mathcal{L}_{N}$, $J_{\alpha,N}(u)=0$ if and only if $u$ is a constant.
\end{theorem}
\begin{remark}
In view of \cite[Proposition 1.5]{GLWY2025}, the constant $\frac{1}{2}$ is sharp when $N$ is even. Proposition \ref{optimalodd} shows $\frac{1}{2}$ is also sharp when $N$ is odd.
\end{remark}

Combining with Theorem \ref{GM2018} and Theorem \ref{szego}, the generalized Chang-Yang conjecture is confirmed in arbitrary dimension $N\geq 1$.

Different from previous works, our argument does not attempt to prove that all critical points are constants. Instead, inspired by the work of Frank and Lieb \cite{FrankLieb2012}, we show that stable critical points must be constants, which is sufficient to prove the generalized Chang-Yang conjecture in every dimension $N\geq 3$.
\begin{theorem}\label{instability}
Let $N\geq 3$ and $\frac{1}{2}\leq \alpha\leq 1$. Then every smooth stable critical point restricted to $\mathcal{L}_{N}$ is constant.
\end{theorem}
The stability here is defined in Definition \ref{def:stability}. This is one of the key ingredients of the paper.

\medskip

For a smooth solution $u$ of \eqref{paneitz}, define
\begin{equation}\label{def:gamma-g-M}
\gamma=\int_{\mathbb{S}^N}e^{Nu}dw, \qquad g:=\frac{e^{Nu}}{\gamma},
\qquad
d\mu:=g\ud w,
\qquad
M_{ij}:=\int_{\mathbb S^N}\xi_i\xi_j\,d\mu(\xi).
\end{equation}

Then the center-of-mass condition becomes
\begin{equation*}
\int_{\mathbb{S}^N}\xi d\mu=0
\end{equation*}

The matrix $M$ is symmetric and positive definite with
\begin{equation*}
\text{tr}M=\int_{\mathbb{S}^N}|\xi|^2d\mu=1.
\end{equation*}

Another key ingredient is the following estimate of eigenvalues of $M$.  
\begin{theorem}\label{eigenest}
Assume $N\geq 3$, $\alpha\geq \frac{1}{2}$. Let $u\in \mathcal{L}_N$ be a smooth solution of \eqref{paneitz}. If $m>\frac{1}{N+1}$ is an eigenvalue of $M$, then
\begin{equation*}
m<\frac{1}{N-1}.
\end{equation*}
\end{theorem}
This eigenvalue estimate on $M$ enables us to construct an unstable direction for non-constant solutions.

\subsection{Outline of the paper}

This paper is organized as follows. In Section \ref{preliminary}, we  establish the existence and regularity of a minimizer and prove some useful integral identities. Estimates of eigenvalues of $M$ as in Theorem \ref{eigenest} are established in Section \ref{weightedl^2est} with an integral representation of $P_N^{-1}$. Section \ref{Sec: instability} proves the rigidity of stabile solutions and proves Theorem \ref{instability}. Proof of Theorem \ref{main} is presented in Section \ref{proofmain}. The optimality of $\alpha=\frac{1}{2}$ is shown in Appendix \ref{optimal}. Appendix \ref{Technicalities} includes the differentiablity of some functionals.

\section{Preliminaries}\label{preliminary}
In this section, we collect some known facts about the equation, and prove some integral identities.

\subsection{Existence and regularity of a minimizer}
In this subsection, we prove that the infimum of $J_{\alpha,N}$ over $\mathcal{L}_{N}$ is attained for $\frac{1}{2}<\alpha<1$.

\begin{proposition}\label{minimizerexist}
Let $N\geq 3$ and $\frac{1}{2}<\alpha<1$. Then  $J_{\alpha,N}$ attains its infimum on $\mathcal{L}_{N}$.
\end{proposition}
\begin{proof}
Since $J_{\alpha,N}$ and $\mathcal{L}_{N}$ are invariant under addition of constants, we may assume
\begin{equation*}
\int_{\mathbb{S}^N}udw=0.
\end{equation*}

Choose $\alpha_0$ with $\frac{1}{2}<\alpha_0<\alpha$. By Chang-Yang \cite[Lemma 4.6]{ChangYang1995} and Wei-Xu \cite[proof of Theorem 2.6]{WX1998}, $J_{\alpha_0,N}(u)\geq -C$ for some $C=C(\alpha_0,N)$ for every $u\in \mathcal{L}_N$. Then we have
\begin{equation}\label{Jsplit}
J_{\alpha,N}(u)=J_{\alpha_0,N}(u)
+\frac{\alpha-\alpha_0}{2}\int_{\mathbb{S}^N}uP_{N}u dw
\geq-C+\frac{\alpha-\alpha_0}{2}\int_{\mathbb{S}^N}uP_{N}u dw.
\end{equation}

Since $J_{\alpha,N}(0)=0$ and $0\in \mathcal{L}_{N}$, we can choose a minimizing sequence $u_j$ with 
\begin{equation*}
\int_{\mathbb{S}^N}u_jdw=0
\end{equation*}
and $J_{\alpha,N}(u_j)\leq 1$. 

By the spectral formula \eqref{PN:eigenvalue}, for any mean-zero function $u$, the quadratic form
\begin{equation*}
\int_{\mathbb{S}^N}uP_{N}udw    
\end{equation*}
is equivalent to $\|u\|_{H^{\frac{N}{2}}(\mathbb{S}^N)}^2$. Then \eqref{Jsplit} implies that $\{u_j\}$ is uniformly bounded in $H^{\frac{N}{2}}(\mathbb{S}^N)$. Hence, up to a subsequence, we have
\begin{equation*}
u_j\rightharpoonup u \text{ in }H^{\frac{N}{2}}(\mathbb{S}^N),\qquad  u_j\to u\text{ in }L^q(\mathbb{S}^N)\text{ for any }1<q<\infty
\end{equation*}
and $u_j\to u$ almost everywhere.

For fixed $p>1$, Beckner's inequality applied to $pu_j$ gives that
\begin{equation*}
\log\int_{\mathbb{S}^N}e^{Npu_j}dw\leq \frac{Np^2}{2(N-1)!}\int_{\mathbb{S}^N}u_j P_{N}u_jdw.
\end{equation*}
Therefore $\{e^{Nu_j}\}$ is uniformly bounded in $L^p(\mathbb{S}^N)$. Since $u_j\to u$ almost everywhere, by Vitali's convergence theorem, we have
\begin{equation*}
e^{Nu_j}\to e^{Nu}\text{ in }L^1(\mathbb{S}^N).
\end{equation*}
Hence the center-of-mass condition is also preserved. By weak lower semi-continuity of the quadratic term $\int_{\mathbb{S}^N}uP_{N}udw$, taking $j\to\infty$, we see that $u$ is a minimizer and $u\in \mathcal{L}_{N}$.
\end{proof}

Next we show that $u$ is a smooth function.
\begin{proposition}\label{minimizerreg}
Let $N\geq 3$ and $\frac{1}{2}\leq\alpha\leq1$. If $u\in H^{\frac{N}{2}}(\mathbb{S}^N)$ is a minimizer of $J_{\alpha,N}$ over $\mathcal{L}_{N}$, then $u\in C^{\infty}(\mathbb{S}^N)$.
\end{proposition}
\begin{proof}
By Lemma \ref{Differentiable}, the constraint $C(u)=\int_{\mathbb{S}^N}\xi e^{Nu}dw$ satisfies
\begin{equation*}
\delta C(u)[h]=N\int_{\mathbb{S}^N}\xi h e^{Nu}dw.
\end{equation*}
Therefore by \eqref{def:gamma-g-M}, The restriction of $\delta C(u)$ to Span$\{x_1,\cdots, x_{N+1}\}$ is the matrix $N\gamma M$, which is positive definite. Thus, 
\begin{equation*}
\mathcal{L}_{N}=\{u\in H^{\frac{N}{2}}(\mathbb{S}^N):\  C(u)=0\}
\end{equation*}
is a regular codimension-$(N+1)$ constraint in $H^{\frac{N}{2}}(\mathbb{S}^N)$. Hence by Proposition \ref{Lagrange}, we see that $u$ is a weak solution of \eqref{paneitz}.

For every fixed $p>1$, applying Beckner's inequality with $\alpha=1$ to $pu$ and $-pu$, we obtain $e^{N|u|}\in L^p(\mathbb{S}^N)$. Thus, the right-hand side of \eqref{axialLagrange} belongs to $L^p(\mathbb{S}^N)$. Since $J_{\alpha,N}$ and $\mathcal{L}_{N}$ are invariant under addition of constants, we may assume
\begin{equation*}
\int_{\mathbb{S}^N}udw=0,
\end{equation*}
Then elliptic estimates give
\begin{equation}\label{elliptic-Lp}
\|u\|_{W^{N,p}(\mathbb{S}^N)} \leq C_{N,p}\left(\|P_Nu\|_{L^p(\mathbb{S}^N)}+\|u\|_{L^p(\mathbb{S}^N)}\right).
\end{equation}

For odd $N$, \eqref{elliptic-Lp} is the standard estimate for the classical elliptic pseudo-differential operator obtained from the complex-power construction \cite{Seeley1967}; see also \cite{Taylor1991}. For even $N$ it is the usual elliptic estimate. Taking $p>N$, we have $u\in C^{N-1,\epsilon}$ for some $\epsilon\in (0,1)$. A standard bootstrap argument then gives $u\in C^{\infty}(\mathbb{S}^N)$. 
\end{proof}

\subsection{Integral identities}
Hereafter, we use $x_i$ to denote the coordinate function on $\mathbb{S}^N$, that is, $x_i(\xi)=\xi_i$ for $\xi\in\mathbb{S}^N$.  We also write $$X_i=\nabla x_i,\qquad
 i=1,\ldots,N+1.$$
The following identities on the round sphere will be used repeatedly:
\begin{equation}\label{spherical:identities}
\Delta x_i=-Nx_i,
\qquad
\langle X_i,X_j\rangle=\delta_{ij}-x_ix_j,
\qquad
\operatorname{div}(x_jX_i)=\delta_{ij}-(N+1)x_ix_j.
\end{equation} 

\begin{lemma}\label{lem:commutator}
For every smooth function $u$ and every $i=1,\ldots,N+1$,
\begin{equation}\label{commutator}
P_N(X_iu)=X_i(P_Nu)-Nx_iP_Nu.
\end{equation}
\end{lemma}

\begin{proof}
 Let $e_i$ denote the $i$-th coordinate	vector of $\mathbb R^{N+1}$. Then we have
	\begin{equation}\label{formula:Xi-pointwise}
		X_i(\xi)=\nabla_{g_{\mathbb{S}^N}}x_i(\xi)=e_i-x_i(\xi)\xi.
	\end{equation}

To describe the flow generated by $X_i$, for a fixed $\xi\in\mathbb S^N$, we introduce
	\begin{equation*}
		a:=x_i(\xi),
		\qquad
		\xi^\perp:=\xi-ae_i,
		\qquad
		D_s(\xi):=\cosh s+a\sinh s.
	\end{equation*}
	Then $D_s(\xi)>\cosh s-|\sinh s|=e^{-|s|}>0$ since $|a|<1$. Thus we can define
	\begin{equation}\label{formula:conformal-flow}
		\Phi_s(\xi)
		:=
		\frac{\xi^\perp+(a\cosh s+\sinh s)e_i}
		{D_s(\xi)},
	\end{equation}
and a direct computation shows that $\Phi_s(\xi)\in\mathbb S^N$.

	Clearly $\Phi_0(\xi)=\xi$. Its $i$-th coordinate is
	\begin{equation}\label{formula:flow-coordinate}
		x_i(\Phi_s(\xi))
		=
		\frac{a\cosh s+\sinh s}{D_s(\xi)}.
	\end{equation}
	Differentiating \eqref{formula:conformal-flow} gives
	\begin{align*}
		\frac{d}{ds}\Phi_s(\xi)
		&=e_i-\frac{a\cosh s+\sinh s}{D_s(\xi)}\Phi_s(\xi)\\
		&=e_i-x_i(\Phi_s(\xi))\Phi_s(\xi)\\
		&=X_i(\Phi_s(\xi)).
	\end{align*}
Thus $\Phi_s$ is precisely the flow generated by $X_i$:
	\begin{equation*}
		\frac{d}{ds}\Phi_s=X_i\circ\Phi_s,
		\qquad
		\Phi_0=\operatorname{id}_{\mathbb S^N}.
	\end{equation*}
    
We next compute the conformal factor of $\Phi_s$. Since $\operatorname{Hess}x_i=-x_i g_{\mathbb S^N}$, the Lie derivative of the metric $g_{\mathbb S^N}$ in the direction of the vector field $X_i$ is
\begin{equation*}
\mathcal L_{X_i}g_{\mathbb S^N}=2\operatorname{Hess}x_i=-2x_i g_{\mathbb S^N}.
\end{equation*}

On the other hand, we set	$g_s:=\Phi_s^*g_{\mathbb{S}^N}$. The standard differentiation formula for pullbacks along a flow gives
	\begin{equation}\label{ode:pullback-metric}
		\frac{d}{ds}g_s=
		\Phi_s^*(\mathcal L_{X_i}g_{\mathbb{S}^N})=
		-2(x_i\circ\Phi_s)g_s.
	\end{equation}
Since
	\begin{equation*}
		\frac{d}{ds}\log D_s(\xi)=
		\frac{a\cosh s+\sinh s}{D_s(\xi)}=
		x_i(\Phi_s(\xi)).
	\end{equation*}
	Solving \eqref{ode:pullback-metric} with the initial condition $g_0=g_{\mathbb{S}^N}$ therefore gives
	\begin{equation}\label{formula:flow-conformal-factor}
		\Phi_s^*g_{\mathbb{S}^N}=e^{2\omega_s}g_{\mathbb{S}^N},
	\end{equation}
	where $\omega_s(\xi):=-\log D_s(\xi)$.	In particular,
	\begin{equation}\label{formula:omega-derivative}
	\dot\omega_0(\xi)
	=-\left.\frac{d}{ds}\log D_s(\xi)	\right|_{s=0}=-x_i(\xi).
	\end{equation}
    
    Naturality under the diffeomorphism $\Phi_s$ and the conformal covariance \eqref{conformal PN g} give
\begin{equation*}
P_N(u\circ\Phi_s)
=e^{N\omega_s}(P_Nu)\circ\Phi_s=D_s^{-N}(P_N u)\circ \Phi_s.
\end{equation*}

Now we differentiate the above formula at $s=0$. Since
	\begin{equation*}
		\left.	\frac{d}{ds}(u\circ\Phi_s)
		\right|_{s=0}=X_iu
	\end{equation*}
	and
	\begin{equation*}
		\left.\frac{d}{ds}
		\bigl((P_Nu)\circ\Phi_s\bigr)
		\right|_{s=0}=X_i(P_Nu),
	\end{equation*}
	we obtain
	\begin{align*}
		P_N(X_iu)=N\dot\omega_0P_Nu+X_i(P_Nu)=X_i(P_Nu)-Nx_iP_Nu,
	\end{align*}
	where we used $\dot\omega_0=-x_i$ from	\eqref{formula:omega-derivative}. This proves \eqref{commutator}.
\end{proof}

Write $G_i:=X_iu$. Then since $X_ig=NgG_i$, Lemma \ref{lem:commutator} gives
\begin{equation}\label{equation:Gi}
\alpha P_NG_i
=N!\bigl(gG_i+x_i(1-g)\bigr).
\end{equation}
We have the following identities for $G_i$.
\begin{lemma}\label{lem:moment-identities}
Let $u$ be a smooth solution of \eqref{paneitz}. Then
\begin{equation}\label{Gi:mean}
\int_{\mathbb S^N}G_i\,d\mu=0,
\end{equation}
\begin{equation}\label{xjGi:dw}
\int_{\mathbb S^N}x_jG_i\ud w
=\frac1{\alpha N}
\left(M_{ij}-\frac{\delta_{ij}}{N+1}\right),
\end{equation}
and
\begin{equation}\label{xjGi:dmu}
\int_{\mathbb S^N}x_jG_i\,d\mu
=\frac{N+1}{N}
\left(M_{ij}-\frac{\delta_{ij}}{N+1}\right).
\end{equation}
\end{lemma}

\begin{proof}
Since $X_ig=NgG_i$, integration by parts and the center condition give
\begin{equation*}
\int_{\mathbb S^N}G_i\,d\mu
=\frac1N\int_{\mathbb S^N}X_ig\ud w
=-\frac1N\int_{\mathbb S^N}g\operatorname{div}X_i\ud w
=\int_{\mathbb S^N}x_i g\ud w=0.
\end{equation*}

Set
\begin{equation*}
q_{ij}:=x_ix_j-\frac{\delta_{ij}}{N+1}.
\end{equation*}
Then $q_{ij}$ is a spherical harmonic of degree two, and
\begin{equation*}
P_Nq_{ij}=(N+1)!q_{ij}.
\end{equation*}
By \eqref{spherical:identities},
\begin{equation*}
\int_{\mathbb S^N}x_jG_i\ud w
=-\int_{\mathbb S^N}u\operatorname{div}(x_jX_i)\ud w
=(N+1)\int_{\mathbb S^N}uq_{ij}\ud w.
\end{equation*}
Testing \eqref{paneitz} against $q_{ij}$ yields
\begin{equation*}
\alpha(N+1)!\int_{\mathbb S^N}uq_{ij}\ud w
=(N-1)!\left(M_{ij}-\frac{\delta_{ij}}{N+1}\right),
\end{equation*}
which proves \eqref{xjGi:dw}.

Finally,
\begin{equation*}
\begin{aligned}
\int_{\mathbb S^N}x_jG_i\,d\mu
&=\frac1N\int_{\mathbb S^N}x_jX_ig\ud w\\
&=-\frac1N\int_{\mathbb S^N}g\operatorname{div}(x_jX_i)\ud w,
\end{aligned}
\end{equation*}
and \eqref{xjGi:dmu} follows from \eqref{spherical:identities}.
\end{proof}

Let $v$ be a unit eigenvector of $M$ with eigenvalue $m$, and define
\begin{equation}\label{def:xXGt}
x=v\cdot\xi,
\qquad
X=\nabla x,
\qquad
G=Xu,
\qquad
t=\frac{(N+1)m-1}{N}.
\end{equation}
Taking the corresponding linear combinations in Lemma
\ref{lem:moment-identities}, we obtain
\begin{equation}\label{basic:G-identities}
\int_{\mathbb S^N}G\,d\mu=0,
\qquad
\int_{\mathbb S^N}xG\,d\mu=t,
\qquad
\int_{\mathbb S^N}xG\ud w=\frac{t}{\alpha(N+1)}.
\end{equation}
Since $0<m<1$, we see that
\begin{equation}\label{range:t}
-\frac1N<t<1.
\end{equation}

\section{Eigenvalue estimates of $M$}\label{weightedl^2est}
In this section, we establish the estimates of the eigenvalues of the matrix $M=(M_{ij})_{1\leq i,j\leq N+1}$ with 
\begin{equation*}
M_{ij}=\int_{\mathbb{S}^N}\xi_i\xi_j d\mu(\xi)   
\end{equation*}
and prove Theorem \ref{eigenest}.

\subsection{Integral representation of $P_N^{-1}$}
For any smooth function $h$ with
\begin{equation}\label{meanzero}
\int_{\mathbb{S}^N}hdw =0,
\end{equation}
define the operator $\mathcal{K}$ by
\begin{equation*}
\mathcal{K}h(\xi):=-\int_{\mathbb{S}^N}\log(1-\xi\cdot \eta )h(\eta)dw(\eta).
\end{equation*}

We need the following Funk--Hecke formula (see \cite[Theorem 2.39]{Mori1998}).
\begin{lemma}\label{KY}
Suppose $K$ is a complex-valued function with
\begin{equation*}
\int_{-1}^{1}|K(t)|(1-t^2)^{\frac{N}{2}-1}dt<\infty.
\end{equation*}
Then for any spherical harmonics $Y_k\in \mathcal{H}_k$, we have
\begin{equation*}
\int_{\mathbb{S}^N}K(\xi\cdot \eta)Y_k(\eta)dw(\eta)=\hat{K}(k)Y_k(\xi),
\end{equation*}
where
\begin{equation*}
\hat{K}(k)=\frac{\Gamma(\frac{N+1}{2})}{\sqrt{\pi}\Gamma(\frac{N}{2})}\int_{-1}^{1}K(t)F_k(t)(1-t^2)^{\frac{N}{2}-1}dt.
\end{equation*}
and $F_k$ is the normalized Gegenbauer polynomials introduced in Appendix \ref{optimal}.
\end{lemma}

By the Funk--Hecke formula above, we have the following characterization of $P_N^{-1}$.
\begin{lemma}\label{KH}
For every smooth function $h$ satisfying \eqref{meanzero}, we have 
\begin{equation}\label{PNKh}
P_N \mathcal{K}h=(N-1)!h.
\end{equation}
Equivalently, on $\mathcal{H}_k$ with $k\geq 1$, the operator $\mathcal{K}$ has eigenvalue
\begin{equation*}
\kappa_k=(N-1)!\frac{\Gamma(k)}{\Gamma(k+N)}.
\end{equation*}
\end{lemma}
\begin{proof}
We first check \eqref{PNKh} for each $Y_k\in \mathcal{H}_k$ with $k\geq 1$.

Applying Lemma \ref{KY} to $K(t):=-\log(1-t)$, we have for every $k\geq 1$ and $Y_k\in \mathcal{H}_k$,
\begin{equation*}
\mathcal{K}Y_k=\kappa_k Y_k
\end{equation*}
with
\begin{equation*}
\kappa_k=-\frac{\Gamma(\frac{N+1}{2})}{\sqrt{\pi}\Gamma(\frac{N}{2})}\int_{-1}^{1}\log (1-z)F_k(z)(1-z^2)^{\frac{N}{2}-1}dz
\end{equation*}

By \eqref{Rodrigues}, for $k\geq 1$, integrating by parts $k$ times gives
\begin{equation*}
\begin{aligned}
\kappa_k&=\frac{\Gamma(\frac{N+1}{2})(k-1)!}{2^k\sqrt{\pi}\Gamma(k+\frac{N}{2})}\int_{-1}^{1}(1-z)^{\frac{N}{2}-1}(1+z)^{k+\frac{N}{2}-1}dz\\
&=(N-1)!\frac{\Gamma(k)}{\Gamma(k+N)}.
\end{aligned}
\end{equation*}
Here all boundary terms vanish, and the last equality follows from the duplication formula for the Gamma function and Beta function.

Write $h=\sum_{k=1}^{\infty}h_k$ with $h_k\in\mathcal H_k$. For finite sums $h^{(L)}=\sum_{k=1}^{L}h_k$, computation above gives
\begin{equation*}
\mathcal{K}h^{(L)}=\sum_{k=1}^{L}
(N-1)!\frac{\Gamma(k)}{\Gamma(k+N)}h_k
\end{equation*}
and hence
\begin{equation*}
P_N\mathcal Kh^{(L)}=(N-1)!h^{(L)}.
\end{equation*}

Since $h$ is smooth, we can pass $L\to \infty$ and get the desired result.

\end{proof}

Let $v$ be any unit eigenvector of $M$ with eigenvalue $m$. For any $\xi,\eta\in \mathbb{S}^N$, write $x=v\cdot \xi$, $y=v\cdot \eta$ and $s=\xi\cdot \eta$. Define the kernel
\begin{equation}
D_v(\xi,\eta):=
\begin{cases}
\frac{(x-y)^2}{1-s},
  &\qquad \xi\neq \eta,\\
0,&\qquad \xi=\eta.
\end{cases}
\end{equation}
Since
\begin{equation*}
(x-y)^2\leq |\xi-\eta|^2=2(1-s),
\end{equation*}
we have $0\leq D_v\leq 2$. 

We have the following identity.
\begin{theorem}\label{iintDVidentity}
Let $u$ be a smooth solution of \eqref{paneitz} and let $t$ be defined in \eqref{def:xXGt}. Then we have
\begin{equation}
\iint_{\mathbb{S}^N\times \mathbb{S}^N}D_v(\xi,\eta)d\mu(\xi)d\mu(\eta)=2(m-\alpha t).
\end{equation}
\end{theorem}

Recall $X=\nabla x$ from \eqref{def:xXGt}. To prove Theorem \ref{iintDVidentity}, we will evaluate
\begin{equation}\label{hxXKh}
\int_{\mathbb{S}^N}hxX(\mathcal{K}h) dw
\end{equation}
in two different ways.

The following lemma gives the first evaluation.

\begin{lemma}\label{lem hxXKh1}
Under the same assumption as in Theorem \ref{iintDVidentity}, we have
\begin{equation}
\int_{\mathbb{S}^N}hxX(\mathcal{K}h) dw=\alpha t-\frac{t}{N+1}.
\end{equation}
\end{lemma}

\begin{proof}
Let $h=g-1$. Then $h$ satisfies \eqref{meanzero}. Lemma \ref{KH} then gives
\begin{equation}\label{PN-1u}
u-\int_{\mathbb{S}^N}u dw=\frac{1}{\alpha}\mathcal{K}h.
\end{equation}

Applying $X=\nabla x$ to both sides and recalling $G=Xu$ from \eqref{def:xXGt}, we obtain
\begin{equation}\label{XKh}
X\mathcal{K}h=\alpha G.
\end{equation}

Hence, by \eqref{basic:G-identities}, we have
\begin{equation}\label{hxXKh1}
\begin{aligned}
\int_{\mathbb{S}^N}hxX(\mathcal{K}h) dw
&=\alpha \int_{\mathbb{S}^N}xGhdw\\
&=\alpha \int_{\mathbb{S}^N}xGd\mu-\alpha \int_{\mathbb{S}^N}xGdw\\
&=\alpha t-\frac{t}{N+1}.
\end{aligned}
\end{equation}
\end{proof}

In the following, we will evaluate \eqref{hxXKh} in a different way. We first establish its integral representation.
\begin{lemma}\label{lem hxXKh2}
In the same assumption as in Theorem \ref{iintDVidentity}, we have
\begin{equation}\label{hxXKh2}
\int_{\mathbb{S}^N}hxX(\mathcal{K}h) dw=-\frac{1}{2}\iint_{\mathbb{S}^N\times \mathbb{S}^N}h(\xi)h(\eta)D_{v}(\xi,\eta)dw(\xi)dw(\eta).
\end{equation}
\end{lemma}
\begin{proof}
For any fixed $\eta\in \mathbb{S}^N$, we have
\begin{equation*}
\nabla_{\xi} (\xi\cdot \eta)=\eta-(\xi\cdot \eta)\xi.
\end{equation*}

Since $X(\xi)=v-x\xi$ and $s=\xi\cdot \eta$, we have
\begin{equation*}
X_{\xi}s=(v-x\xi)\cdot \eta=y-xs.
\end{equation*}
Thus, for $\xi\neq \eta$, we have
\begin{equation*}
X_\xi [-\log (1-s)]=\frac{y-xs}{1-s}.
\end{equation*}

For $0<\epsilon<1$, define the operator $\mathcal{K}_{\epsilon}$ by
\begin{equation*}
\mathcal{K}_{\epsilon}h(\xi):=-\int_{\mathbb{S}^N}\log(1-\xi\cdot \eta +\epsilon)h(\eta)dw(\eta).
\end{equation*}
Since $h$ is smooth, $\mathcal{K}_{\epsilon}h$ is also smooth. Then we have
\begin{equation*}
X(\mathcal{K}_{\epsilon}h)(\xi)=\int_{\mathbb{S}^N}\frac{y-xs}{1-s+\epsilon}h(\eta)dw(\eta).   
\end{equation*}

For any $\xi\neq \eta$, we have
\begin{equation*}
|\log(1-\xi\cdot \eta +\epsilon)|\leq C(1+|\log|\xi-\eta||).
\end{equation*}
    
Since $X(\xi)\in T\mathbb{S}^N$ and $|X(\xi)|\leq 1$, we also have
\begin{equation*}
\frac{|y-xs|}{1-s+\epsilon}=\frac{|X_{\xi}s|}{1-s+\epsilon}\leq \frac{|\nabla_{\xi}s|}{1-s}=\sqrt{\frac{1+s}{1-s}}\leq \frac{2}{|\xi-\eta|}.
\end{equation*}
The right-hand sides of the two inequalities above are integrable in $\eta$.

Define 
\begin{equation*}
A_{\epsilon}(\xi,\eta)=
\begin{cases}
\frac{y-xs}{1-s+\epsilon},
&\qquad \xi\neq \eta,\\
0,&\qquad \xi=\eta
\end{cases}
\end{equation*}
and
\begin{equation*}
A(\xi,\eta)=
\begin{cases}
\frac{y-xs}{1-s},
&\qquad \xi\neq \eta,\\
0,&\qquad \xi=\eta.
\end{cases}
\end{equation*}

By the dominated convergence theorem, we have 
\begin{equation}\label{Khconv}
\|\mathcal{K}_{\epsilon}h-\mathcal{K}h\|_{L^1(\mathbb{S}^N)}\leq \iint_{\mathbb{S}^N}|\log(1-\xi\cdot \eta +\epsilon)-\log(1-\xi\cdot \eta)||h(\eta)|dw(\xi)dw(\eta)\to 0
\end{equation}
and
\begin{equation}\label{XKhconv}
\|X(\mathcal{K}_{\epsilon}h)-F\|_{L^1(\mathbb{S}^N)}\leq \iint_{\mathbb{S}^N}|A_{\epsilon}-A||h(\eta)|dw(\xi)dw(\eta)\to 0
\end{equation}
as $\epsilon\to 0^+$, where
\begin{equation*}
F(\xi):=\int_{\mathbb{S}^N}A(\xi,\eta)h(\eta)dw(\eta)
\end{equation*}
is absolutely convergent for almost every $\xi\in \mathbb{S}^N$ and belongs to $L^1(\mathbb{S}^N)$.

Furthermore, by \eqref{Khconv} and \eqref{XKhconv}, we have $F=X(\mathcal{K}h)$ in the sense of distribution. In view of \eqref{XKh}, $\mathcal{K}h$ is smooth and
\begin{equation*}
X(\mathcal{K}h)(\xi)=\int_{\mathbb{S}^N}A(\xi,\eta)h(\eta)dw(\eta)=\alpha G(\xi)
\end{equation*}
for almost every $\xi\in \mathbb{S}^N$.

For fixed $\epsilon>0$, since the kernel is smooth, by Fubini's theorem, we have
\begin{equation*}
\int_{\mathbb{S}^N}hxX(\mathcal{K}_{\epsilon}h)dw=\iint_{\mathbb{S}^N\times \mathbb{S}^N}h(\xi)h(\eta)x\frac{y-xs}{1-s+\epsilon}dw(\xi)dw(\eta).
\end{equation*}

Interchanging $\xi$ and $\eta$ gives
\begin{equation*}
\int_{\mathbb{S}^N}hxX(\mathcal{K}_{\epsilon}h)dw=\iint_{\mathbb{S}^N\times \mathbb{S}^N}h(\xi)h(\eta)y\frac{x-ys}{1-s+\epsilon}dw(\xi)dw(\eta).
\end{equation*}

Averaging the two identities, we have
\begin{equation*}
\begin{aligned}
\int_{\mathbb{S}^N}hxX(\mathcal{K}_{\epsilon}h)dw
&=\frac{1}{2}\iint_{\mathbb{S}^N\times \mathbb{S}^N}h(\xi)h(\eta)\frac{(1-s)(x^2+y^2)}{1-s+\epsilon}dw(\xi)dw(\eta)\\
&-\frac{1}{2}\iint_{\mathbb{S}^N \times\mathbb{S}^N}h(\xi)h(\eta)\frac{(x-y)^2}{1-s+\epsilon}dw(\xi)dw(\eta).
\end{aligned}
\end{equation*}

Since $0\leq \frac{1-s}{1-s+\epsilon}\leq 1$ and $0\leq \frac{(x-y)^2}{1-s+\epsilon}\leq \frac{(x-y)^2}{1-s}\leq 2$, the dominated convergence theorem gives
\begin{equation*}
\begin{aligned}
\int_{\mathbb{S}^N}hxX(\mathcal{K}h)dw
&=\frac{1}{2}\iint_{\mathbb{S}^N\times \mathbb{S}^N}h(\xi)h(\eta)(x^2+y^2)dw(\xi)dw(\eta)\\
&-\frac{1}{2}\iint_{\mathbb{S}^N\times \mathbb{S}^N}h(\xi)h(\eta)D_v(\xi,\eta)dw(\xi)dw(\eta).
\end{aligned}
\end{equation*}
By \eqref{meanzero}, the first integral vanishes and we get the desired result.
\end{proof}

Using the preceding integral representation, we give the second evaluation of \eqref{hxXKh}.
\begin{lemma}
Under the same assumption as in Theorem \ref{iintDVidentity}, we have
\begin{equation*}
\int_{\mathbb{S}^N}hxX(\mathcal{K}h) dw=\frac{1+(N-1)t}{N+1}-\frac{1}{2}\iint_{\mathbb{S}^N\times \mathbb{S}^N}D_vd\mu d\mu.
\end{equation*}
\end{lemma}
\begin{proof}
Fix any $\xi\in \mathbb{S}^N$ and let $x=v\cdot \xi$. By rotational invariance,
we can choose coordinates such that $\xi=e_{N+1}$ and $v=\sqrt{1-x^2}e_1+xe_{N+1}$.

Writing $\eta=(\eta_1,\dots,\eta_{N+1})$, we have $s=\eta_{N+1}$ and $y=\sqrt{1-x^2}\eta_1+x\eta_{N+1}$. Therefore, for $\eta\neq \xi$, we have
\begin{equation*}
D_v(\xi,\eta)=x^2(1-\eta_{N+1})-2x\sqrt{1-x^2}\eta_1+(1-x^2)\frac{\eta_1^2}{1-\eta_{N+1}}.
\end{equation*}

Moreover, by symmetry, we have
\begin{equation*}
N\int_{\mathbb{S}^N}\frac{\eta_1^2}{1-\eta_{N+1}}dw(\eta)=\int_{\mathbb{S}^N}\frac{\sum_{i=1}^{N}\eta_i^2}{1-\eta_{N+1}}dw(\eta)=\int_{\mathbb{S}^N}(1+\eta_{N+1})dw(\eta)=1.
\end{equation*}

Thus, we have
\begin{equation}\label{intDv}
\int_{\mathbb{S}^N}D_v(\xi,\eta)dw(\eta)=x^2+\frac{1-x^2}{N}=\frac{1}{N}+\frac{N-1}{N}x^2
\end{equation}
and
\begin{equation}\label{iintDv}
\iint_{\mathbb{S}^N\times \mathbb{S}^N}D_v(\xi,\eta)dw(\eta)dw(\xi)=\frac{1}{N}+\frac{N-1}{N}\int_{\mathbb{S}^N}x^2dw=\frac{2}{N+1}.
\end{equation}

Since
\begin{equation*}
\int_{\mathbb{S}^N}x^2d\mu=m,
\end{equation*}
by \eqref{hxXKh2}, \eqref{intDv} and \eqref{iintDv}, we have
\begin{equation*}
\begin{aligned}
\int_{\mathbb{S}^N}hxX(\mathcal{K}h) dw
&=
-\frac{1}{2}\iint_{\mathbb{S}^N\times \mathbb{S}^N}h(\xi)h(\eta)D_{v}(\xi,\eta)dw(\xi)dw(\eta)\\
&=\frac{1}{N}+\frac{N-1}{N}m-\frac{1}{N+1}-\frac{1}{2}\iint_{\mathbb{S}^N\times \mathbb{S}^N}D_vd\mu d\mu\\
&=\frac{1+(N-1)t}{N+1}-\frac{1}{2}\iint_{\mathbb{S}^N\times \mathbb{S}^N}D_vd\mu d\mu.    
\end{aligned}
\end{equation*}
\end{proof}

With the two evaluations, we can give a proof of Theorem \ref{iintDVidentity}.
\begin{proof}[Proof of Theorem \ref{iintDVidentity}]

Comparing the two evaluations with \eqref{hxXKh1} and using
\begin{equation*}
m=\frac{1+Nt}{N+1},
\end{equation*}
we get
\begin{equation*}
\iint_{\mathbb{S}^N\times \mathbb{S}^N}D_vd\mu d\mu=2(m-\alpha t)
\end{equation*}
as desired.
\end{proof}

We give the following estimate.
\begin{lemma}\label{Dvlower}
Suppose $v$ is a unit eigenvector of $M$ corresponding to eigenvalue $m$, then
\begin{equation*}
\iint_{\mathbb{S}^N\times \mathbb{S}^N}D_v(\xi,\eta)d\mu(\xi)d\mu(\eta)>\frac{2m}{m+1}.
\end{equation*}
\end{lemma}
\begin{proof}
Recall that $x=v\cdot \xi$, $y=v\cdot \eta$ and $s=\xi\cdot \eta$.

Cauchy--Schwarz inequality gives that
\begin{equation}\label{x-yCauchy}
\left(\iint_{\mathbb{S}^N\times \mathbb{S}^N}(x-y)^2d\mu(\xi)d\mu(\eta)\right)^2\leq \iint_{\mathbb{S}^N \times\mathbb{S}^N}\frac{(x-y)^2}{1-s}d\mu(\xi)d\mu(\eta) \iint_{\mathbb{S}^N\times\mathbb{S}^N}(x-y)^2(1-s)d\mu(\xi)d\mu(\eta).
\end{equation}

For the left-hand side, since
\begin{equation*}
\int_{\mathbb{S}^N}xd\mu=v\cdot \int_{\mathbb{S}^N}\xi d\mu=0
\end{equation*}
and
\begin{equation*}
\int_{\mathbb{S}^N}x^2d\mu=v\cdot Mv=m,
\end{equation*}
we have
\begin{equation}\label{Cauchy1}
\iint_{\mathbb{S}^N\times \mathbb{S}^N}(x-y)^2d\mu(\xi)d\mu(\eta)=2m
\end{equation}

For the right-hand side, since
\begin{equation*}
\iint_{\mathbb{S}^N \times \mathbb{S}^N}x^2sd\mu(\xi)d\mu(\eta)=\int_{\mathbb{S}^N}x^2\xi\cdot \left(\int_{\mathbb{S}^N}\eta d\mu(\eta)\right)d\mu(\xi)=0
\end{equation*}
and
\begin{equation*}
\iint_{\mathbb{S}^N\times \mathbb{S}^N}xysd\mu(\xi)d\mu(\eta)=\sum_{j=1}^{N+1}\left(\int_{\mathbb{S}^N}x\xi_j d\mu(\xi)\right)\left(\int_{\mathbb{S}^N}y\eta_j d\mu(\eta)\right)=\sum_{j=1}^{N+1}(Mv)_j^2=m^2,
\end{equation*}
we have
\begin{equation}\label{Cauchy2}
\iint_{\mathbb{S}^N}(x-y)^2(1-s)d\mu(\xi)d\mu(\eta)=2m(m+1).
\end{equation}

Plugging \eqref{Cauchy1} and \eqref{Cauchy2} into \eqref{x-yCauchy}, we get
\begin{equation*}
\iint_{\mathbb{S}^N\times \mathbb{S}^N}D_v(\xi,\eta)d\mu(\xi)d\mu(\eta)\geq\frac{2m}{m+1}.
\end{equation*}

If the equality holds, then we have
\begin{equation*}
\frac{x-y}{\sqrt{1-s}}=c(x-y)\sqrt{1-s}
\end{equation*}
for $d\mu(\xi)d\mu(\eta)$-almost every $(\xi,\eta)\in \mathbb{S}^N\times \mathbb{S}^N$ for some constant $c$.

By \eqref{Cauchy1}, the set $\{x\neq y\}$ has positive measure with respect to $d\mu(\xi)d\mu(\eta)$. It follows that $c\neq 0$ and 
\begin{equation*}
1-\xi\cdot \eta=c^{-1}
\end{equation*}
for $d\mu(\xi)d\mu(\eta)$-almost every $(\xi,\eta)\in \mathbb{S}^N\times \mathbb{S}^N$, which is impossible since the measure $\mu$ has a positive density. Hence we get the desired strict inequality.
\end{proof}

With the estimates above, we can prove Theorem \ref{eigenest}.
\begin{proof}[Proof of Theorem \ref{eigenest}]
Suppose $m>\frac{1}{N+1}$ and $t=\frac{1}{N}((N+1)m-1)>0$. By Theorem \ref{iintDVidentity} and Lemma \ref{Dvlower}, we have
\begin{equation*}
2(m-\alpha t)>\frac{2m}{m+1}.
\end{equation*}

Since $\alpha \geq \frac{1}{2}$ and $t>0$, we have
\begin{equation*}
\frac{t}{2}<\frac{m^2}{m+1}.
\end{equation*}
Plugging in $t=\frac{1}{N}((N+1)m-1)$, we get
\begin{equation*}
(N-1)m^2-Nm+1=(m-1)((N-1)m-1)>0.
\end{equation*}

Since $0<m<1$, we get $m<\frac{1}{N-1}$ as desired.
\end{proof}

\section{Rigidity of stable critical points}\label{Sec: instability}
In this section, we prove Theorem \ref{instability}. That is, every stable critical point of ${J}_{\alpha,N}$ restricted to $\mathcal{L}_N$ is constant.

We first study the second variation of $J_{\alpha,N}$. Recall that the probability measure
\begin{equation*}
d\mu=\frac{e^{Nu}}{\gamma}dw
\end{equation*}
on $\mathbb{S}^N$ satisfies
\begin{equation*}
\int_{\mathbb{S}^N}\xi d\mu=0.
\end{equation*}

Differentiating $J_{\alpha,N}$ twice gives, for any smooth functions $h_1$ and $h_2$, 
\begin{equation}\label{secondvariationJ}
\begin{aligned}
\delta^2 J_{\alpha,N}(u)[h_1,h_2]
=&\alpha \int_{\mathbb{S}^N}h_1P_{N}h_2dw-N!\left(\int_{\mathbb{S}^N}h_1h_2d\mu-\int_{\mathbb{S}^N}h_1d\mu\int_{\mathbb{S}^N}h_2d\mu\right).
\end{aligned}
\end{equation}

The tangent space of the constraint set $\mathcal{L}_N$ at $u$ is given by
\begin{equation*}
T_u\mathcal{L}_N=\left\{h:\ \int_{\mathbb{S}^N}\xi_jh d\mu=0,\ j=1,\dots, N+1\right\}.
\end{equation*}

\begin{definition}\label{def:stability}
A critical point $u$ of $J_{\alpha,N}$ restricted to $\mathcal{L}_N$ is called stable if
\begin{equation*}
\delta^2 J_{\alpha,N}(u)[h,h]\geq 0
\end{equation*}
for any smooth function $h\in T_u\mathcal{L}_{N}$.
\end{definition}

Recall the constraint functional $C(u):=\int_{\mathbb{S}^N}\xi e^{Nu}dw$. Then $\mathcal{L}_N=C^{-1}(0)$.

Let $h\in T_u\mathcal{L}_N$ be a smooth function on $\mathbb{S}^N$. For $s\in\R$ and $r\in\R^{N+1}$, consider $\Psi(s,r):=C(u+sh+r\cdot \xi)$. For fixed smooth functions $u$ and $h$, $\Psi$ is a $C^{\infty}$ function in a neighborhood of $(0,0)$ and
\begin{equation*}
D_r\Psi(0,0)=N\gamma M,
\end{equation*}
which is positive definite. Then the implicit function theorem gives a smooth function $r=r(s)$, defined for small $s$, such that $r(0)=0$,
\begin{equation*}
u_s=u+sh+r(s)\cdot \xi
\end{equation*}
and $C(u_s)=0$.

\begin{lemma}\label{delta2J}
For $u_s$ defined above, we have
\begin{equation}\label{ds^2 J al N}
\left.\frac{d^2}{ds^2}\right|_{s=0}J_{\alpha,N}(u_s)=\delta^2J_{\alpha,N}(u)[h,h].
\end{equation}
\end{lemma}
\begin{proof}
Differentiating $C(u_s)$ at $s=0$ and using $h\in T_u\mathcal{L}_N$, we have
\begin{equation*}
0=\delta C(u)[h]+r'(0)\delta C(u)[x]=N\int_{\mathbb{S}^N}e^{Nu}dw\int_{\mathbb{S}^N}xhd\mu+r'(0)\delta C(u)[x]=r'(0)\delta C(u)[x].
\end{equation*}
Thus $r'(0)=0$ and $\dot u_0=h$. Here $\dot u_0$ denotes the derivative of $u_s$ in $s$ at $s=0$.

	Differentiating $J_{\alpha,N}(u_s)$ twice and evaluating at $s=0$ yields
	\begin{align*}
		\left.
		\frac{d^2}{ds^2}
		\right|_{s=0}
		\mathcal J_{\alpha,N}(u_s)
		={}&
		\int_{\mathbb S^N}
		\ddot u_0
		\left[
		\alpha P_Nu+(N-1)!(1-g)
		\right]dw\\
		&+
		\alpha\int_{\mathbb S^N}hP_Nh\,dw
			-N!
		\left[
		\int_{\mathbb S^N}h^2g\,dw
		-
		\left(
		\int_{\mathbb S^N}hg\,dw
		\right)^2
		\right].
	\end{align*}
   The first line on the above right-hand side vanishes since $u$ satisfies \eqref{paneitz}. Therefore \eqref{ds^2 J al N} follows from the definition \eqref{secondvariationJ}.
\end{proof}

\begin{lemma}\label{secondvariation}
Let $v,m,x,G,t$ be defined as in \eqref{def:xXGt} and let
\begin{equation*}
\phi:=G-\frac{t}{m}x=G-\frac{(N+1)t}{1+Nt}x.
\end{equation*}
Then we have $\phi \in T_u\mathcal{L}_N$,
\begin{equation*}
\int_{\mathbb{S}^N}\phi d\mu=0
\end{equation*}
and
\begin{equation*}
\begin{aligned}
\frac{1}{N!}\delta^2J_{\alpha,N}(u)[\phi,\phi]=H_N(\alpha,t)
\end{aligned}
\end{equation*}
with
\begin{equation*}
H_N(\alpha,t):=\frac{t\,[1+Nt-\alpha(N+1)]
[1+t(N-\alpha(N+1))]}
{\alpha(N+1)(1+Nt)^2}.
\end{equation*}
Furthermore, we have
\begin{equation}\label{sign}
\text{sgn}H_N(\alpha,t)=\text{sgn}(t(m-\alpha)).
\end{equation}
\end{lemma}

\begin{proof}
For every $a\in\R^{N+1}$, Lemma \ref{lem:moment-identities} and
$Mv=mv$ give
\begin{equation*}
\int_{\mathbb S^N}(a\cdot \xi)G\,d\mu=t(a\cdot v),
\qquad
\int_{\mathbb S^N}(a\cdot \xi)x\,d\mu=m(a\cdot v).
\end{equation*}

Then we have
\begin{equation*}
\int_{\mathbb{S}^N}\xi_j\phi d\mu=0
\end{equation*}
for $j=1,\dots N+1$.

By \eqref{basic:G-identities}, we get
\begin{equation*}
\int_{\mathbb{S}^N}\phi d\mu=0.
\end{equation*}

Taking the linear combination of \eqref{equation:Gi} in the direction
$v$ and testing against $G$, we obtain
\begin{equation}\label{G:energy}
\frac{\alpha}{N!}\int_{\mathbb S^N}GP_NG\ud w
=\int_{\mathbb S^N}G^2\,d\mu
-t+\frac{t}{\alpha(N+1)}.
\end{equation}

Since $P_{N}x=N!x$ and
\begin{equation*}
\int_{\mathbb{S}^N}x^2 dw=\frac{1}{N+1},\qquad \int_{\mathbb{S}^N}xG dw=\frac{t}{(N+1)\alpha},
\end{equation*}
we find
\begin{equation}\label{phiPrNphi}
\frac{\alpha}{N!}\int_{\mathbb{S}^N}\phi P_{N}\phi dw=\int_{\mathbb{S}^N}G^2d\mu-t+\frac{t}{\alpha (N+1)}-\frac{2t^2}{1+Nt}+\frac{\alpha (N+1)t^2}{(1+Nt)^2}.
\end{equation}
On the other hand, we have
\begin{equation}\label{phi2}
\int_{\mathbb{S}^N}\phi^2 d\mu=\int_{\mathbb{S}^N}G^2d\mu-\frac{(N+1)t^2}{1+Nt}.
\end{equation}

Subtracting \eqref{phi2} from \eqref{phiPrNphi}, and using equation \eqref{secondvariationJ}, we have
\begin{align*}
\frac{1}{N!}\delta^2J_{\alpha,N}(u)[\phi,\phi]=&-t+\frac{t}{\alpha (N+1)}-\frac{2t^2}{1+Nt}+\frac{\alpha (N+1)t^2}{(1+Nt)^2}+\frac{(N+1)t^2}{1+Nt}\\
=& \frac{t\,[1+Nt-\alpha(N+1)]
[1+t(N-\alpha(N+1))]}
{\alpha(N+1)(1+Nt)^2}\\
=&H_N(\alpha,t).
\end{align*}    

By \eqref{def:xXGt}, we have 
\begin{equation*}
1+Nt-\alpha (N+1)=(N+1)(m-\alpha).
\end{equation*}
It remains to show $1+t(N-\alpha(N+1))>0$. Indeed, let $c=N-\alpha(N+1)$. If $c\geq 0$, by $t\geq -\frac{1}{N}$, we have
\begin{equation*}
1+ct\ge 1-\frac{c}{N}=\frac{N+1}{N}\alpha>0.
\end{equation*}

If $c<0$ and $\alpha<1$, then $t<1$ gives 
\begin{equation*}
1+ct>1+c=(N+1)(1-\alpha)>0.
\end{equation*}

If $\alpha=1$, then $c=-1$ and $1+ct=1-t>0$. This proves \eqref{sign}.
\end{proof}

With the lemma above, we can prove the rigidity of stable critical points when $M=\frac{1}{N+1}I$.

\begin{lemma}\label{M=I/N+1}
Assume $\frac{1}{2}\leq \alpha\leq 1$. If a smooth stable critical point $u$ satisfies
\begin{equation*}
M=\frac{1}{N+1}I,
\end{equation*}
then $u$ is a constant.
\end{lemma}
\begin{proof}
If $M=\frac{1}{N+1}I$, then \eqref{xjGi:dmu} implies that $G_i\in T_u\mathcal{L}_N$.

By \eqref{def:xXGt}, we also have $t=0$. Hence, by Lemma \ref{secondvariation}, we have $\delta^2 J_{\alpha, N}(u)[G_i,G_i]=0$.

By stability of $u$, $\delta^2 J_{\alpha,N}(u)$ is nonnegative on $T_u\mathcal{L}_N$. In particular, we have for any smooth $h\in T_u\mathcal{L}_N$, 
\begin{equation*}
\delta^2 J_{\alpha,N}(u)[G_i+sh,G_i+sh]\geq 0.
\end{equation*}
Differentiating at $s=0$ gives
\begin{equation}\label{Gih1}
\delta^2 J_{\alpha, N}(u)[G_i,h]=0    
\end{equation}
for every smooth $h\in T_u\mathcal{L}_N$.

Using \eqref{equation:Gi}, \eqref{Gi:mean} and the definition of $T_u\mathcal{L}_N$, we get
\begin{equation}\label{Gih2}
\delta^2 J_{\alpha, N}(u)[G_i,h]=N!\int_{\mathbb{S}^N}x_i(1-g)hdw. 
\end{equation}

Let $g_i:=x_i(1-g)$. Then we have
\begin{equation*}
\int_{\mathbb{S}^N}g_i hdw=0.
\end{equation*}

For any $\psi \in C^{\infty}(\mathbb{S}^N)$, define
\begin{equation*}
b_j:=\int_{\mathbb{S}^N}x_j\psi d\mu
\end{equation*}
and 
\begin{equation*}
h:=\psi-(N+1)\sum_{j=1}^{N+1}b_jx_j.
\end{equation*}
Then we have
\begin{equation*}
\int_{\mathbb{S}^N}x_k hd\mu=b_k-(N+1)\sum_{j=1}^{N+1}b_jM_{kj}=b_k-\sum_{j=1}^{N+1}b_j\delta_{kj}=0
\end{equation*}
Hence, $h\in T_u\mathcal{L}_N$. Therefore,
\begin{equation*}
\int_{\mathbb{S}^N}g_i\psi dw=(N+1)\sum_{j=1}^{N+1}b_j\int_{\mathbb{S}^N}g_i x_j dw.
\end{equation*}

Since
\begin{equation*}
\int_{\mathbb{S}^N}g_i x_jdw=\int_{\mathbb{S}^N}x_i x_j(1-g)dw=\frac{\delta_{ij}}{N+1}-M_{ij}=0.
\end{equation*}
It follows that
\begin{equation*}
\int_{\mathbb{S}^N}g_i\psi dw=0
\end{equation*}
for every smooth $\psi$. Hence
\begin{equation*}
g_i=x_i(1-g)=0
\end{equation*}
for $i=1,\dots, N+1$. 

Multiplying by $x_i$ and summing, we have
\begin{equation*}
1-g=\sum_{i=1}^{N+1}x_i^2(1-g)=0.
\end{equation*}
Thus $g\equiv 1$. By definition of $g$, $u$ is constant.
\end{proof}

We now give the proof of Theorem \ref{instability}.
\begin{proof}[Proof of Theorem \ref{instability}]
If $M=\frac{1}{N+1}I$, the conclusion follows from Lemma \ref{M=I/N+1}.

If $M\neq \frac{1}{N+1}I$, since $\text{tr} M=1$, $M$ has an eigenvalue $m>\frac{1}{N+1}$. By Theorem \ref{eigenest}, we have
\begin{equation*}
m<\frac{1}{N-1}\leq \frac{1}{2}\leq \alpha.
\end{equation*}

Since $t>0$ and $m-\alpha<0$, by Lemma \ref{secondvariation}, there exists a $\phi\in T_u\mathcal{L}_N$ such that
\begin{equation*}
\delta^2J_{\alpha,N}(u)[\phi,\phi]<0,
\end{equation*}
a contradiction to the stability of $u$. Hence $M=\frac{1}{N+1}I$. Consequently, we finish the proof of Theorem \ref{instability}.
\end{proof}

\section{Proof of Theorem \ref{main}}\label{proofmain}
In this section, we complete the proof of Theorem \ref{main}.

\begin{proof}[Proof of Theorem \ref{main}]
First suppose $\frac12<\alpha<1$.  By Proposition \ref{minimizerexist}, Proposition \ref{minimizerreg} and Proposition \ref{Lagrange}, the infimum is attained by a smooth function $u$ solving \eqref{paneitz}. 

A minimizer is stable within the constraint.  Theorem \ref{instability} therefore forces $u$ to be constant, and hence
\begin{equation*}
\inf_{u\in \mathcal{L}_{N}}J_{\alpha,N}(u)=0
\end{equation*}
for any $\frac{1}{2}<\alpha<1$. For any fixed $u\in\mathcal{L}_{N}$, continuity in $\alpha$ gives
\begin{equation*}
J_{\frac{1}{2},N}(u)=\lim_{\alpha\downarrow \frac{1}{2}}J_{\alpha,N}(u)\geq0.
\end{equation*}

If $\alpha\geq1$, Beckner's inequality and non-negativity of $P_N$ give
\begin{equation}\label{alpha>1}
J_{\alpha,N}(u)
=J_{1,N}(u)
+\frac{\alpha-1}{2}\int_{\mathbb{S}^N}uP_{N}u dw
\geq0.
\end{equation}
Constants show that the infimum is zero in every case.

If $\frac12\leq\alpha\leq1$ and
$J_{\alpha,N}(u)=0$, then $u\in \mathcal{L}_{N}$ is a global minimizer. Lemma \ref{Differentiable} and the regularity in Proposition \ref{minimizerreg} show that $u$ is smooth and stable. Theorem \ref{instability} forces it to be constant.  If $\alpha>1$, equality in \eqref{alpha>1} forces
\begin{equation*}
\int_{\mathbb{S}^N}uP_{N}udw=0.
\end{equation*}
Since $P_{N}$ is a nonnegative operator, with kernel equal to the constants, $u$ has to be constant.  

Conversely, the equality is attained by any constant function. This completes the proof of Theorem \ref{main}.
\end{proof}

\appendix

\section{Optimality of $\alpha=\frac{1}{2}$ for odd $N$}\label{optimal}

In this appendix, we provide the proof of optimality of $\alpha=\frac{1}{2}$ for odd $N$, using axially-symmetric functions. 
\begin{proposition}\label{optimalodd}
Let $N\geq 3$ be an odd integer. If $\mathcal{I}_{\alpha,N}(u)\ge 0$ for all $u\in \mathcal{L}_{r,N}$, then $\alpha\ge \frac{1}{2}$.
\end{proposition}

Throughout this appendix, every $O(1)$ term is uniform as $t\to 1^-$.

For axially symmetric functions $u$ of $x=\xi_1\in[-1,1]$, $P_N$ can be written as
\begin{equation*}
P_{r,N}u=   (-1)^{\frac{N}{2}}[(1-x^2)^{\frac{N}{2}}u']^{(N-1)} 
\end{equation*}
when $N$ is even and 
\begin{equation*}
\begin{aligned}
P_{r,N}u
=&
\left(-(1-x^2)\frac{d^2}{dx^2}+Nx\frac{d}{dx}+(\frac{N-1}{2})^2\right)^{\frac{1}{2}}\\
&\cdot \prod_{k=0}^{\frac{N-3}{2}}\left(-(1-x^2)\frac{d^2}{dx^2}+Nx\frac{d}{dx}+k(N-k-1)\right)u    
\end{aligned}
\end{equation*}
when $N$ is odd.

The Beckner functional $J_{\alpha,N}$ reduces to
\begin{equation*}
\begin{aligned}
\mathcal{I}_{\alpha,N}(u)
&=\frac{\alpha}{2}\int_{-1}^{1}(1-x^2)^{\frac{N-2}{2}}uP_{r,N}u \ d x+(N-1)!\int_{-1}^{1}(1-x^2)^{\frac{N-2}{2}}u\ d x\nonumber\\
&-\frac{(N-1)!\sqrt{\pi}\Gamma(\frac{N}{2})}{N\Gamma(\frac{N+1}{2})}\log\left(\frac{\Gamma(\frac{N+1}{2})}{\sqrt{\pi}\Gamma(\frac{N}{2})}\int_{-1}^{1}(1-x^2)^\frac{N-2}{2} e^{Nu} d x\right)    
\end{aligned}
\end{equation*}
and
the zero-center-of-mass restriction set becomes
\begin{equation}\label{AxialLrN}
\mathcal{L}_{r,N}=\left\{u\in H^{\frac{N}{2}}(\mathbb{S}^N):\ u=u(x)\text{ and }\int_{-1}^{1}x(1-x^2)^{\frac{N-2}{2}}e^{Nu} d x=0\right\}.
\end{equation}

Furthermore,  a critical point $u$ of $\mathcal I_{\alpha,N}$ restricted to the set $\mathcal{L}_{r,N}$ solves
\begin{equation}\label{axialLagrange}
\alpha P_{r,N}u+(N-1)!-\frac{(N-1)!}{\gamma}e^{Nu}=\lambda xe^{Nu} \qquad \text{ in }(-1,1),
\end{equation}
where
\begin{equation*}
\gamma=\int_{\mathbb{S}^N}e^{Nu} dw
\end{equation*}
and $\lambda\in \mathbb{R}$ is a Lagrange multiplier. 

It is shown in Proposition \ref{Lagrange} that the Lagrange multiplier $\lambda$ vanishes. Then $u$ solves
\begin{equation}\label{axial}
\alpha P_{r,N}u+(N-1)!-\frac{(N-1)!}{\gamma}e^{Nu}=0 \qquad \text{ in }(-1,1).
\end{equation}

Axially-symmetric spherical harmonics are Gegenbauer polynomials (see \cite[Section 8.93]{GR2007}). Let $F_k$ be the normalized Gegenbauer polynomial with $F_k(1)=1$. With $\nu=\frac{N-1}{2}$, Rodrigues' formula gives
\begin{equation}\label{Rodrigues}
F_k(x)=\frac{(-1)^k\Gamma(\frac{N}{2})}{2^k\Gamma(k+\frac{N}{2})}(1-x^2)^{-\frac{N}{2}+1}\frac{d^k}{dx^k} (1-x^2)^{k+\frac{N}{2}-1}.
\end{equation}

The function $F_k$ satisfies the orthogonality condition
\begin{equation}\label{Fknorm}
\begin{aligned}
    \int_{-1}^{1}(1-x^2)^{\frac{N-2}{2}}F_{k}F_l dx&=\frac{2^{N-1}\Gamma(\frac{N}{2})^2 \Gamma(k+1)}{(k+N-2)!(2k+N-1)}\delta_{kl}.
\end{aligned}    
\end{equation}
Furthermore, it is an eigenfunction of $P_{r,N}$:
\begin{equation}\label{Paneitzeigen}
P_{r,N} F_k=
\begin{cases}
\frac{\Gamma(k+N)}{\Gamma(k)}F_k,&\qquad k\geq 1,\\
0,&\qquad k=0.
\end{cases}
\end{equation}

\begin{proof}[Proof of Proposition \ref{optimalodd}]
We argue by contradiction. Assume for contradiction that $\alpha<\frac{1}{2}$. For $0<t<1$, define
\begin{equation*}
\varphi_t^{\pm}(x):=\log\frac{1-t^2}{1+t^2\mp 2tx}
\end{equation*}
and
\begin{equation*}
T_t(x)=\frac{(1+t^2)x-2t}{1+t^2-2tx}.
\end{equation*}

Then $T_t$ is a diffeomorphism from $[-1,1]$ to $[-1,1]$ and
\begin{equation*}
(1-x^2)^{\frac{N}{2}-1}e^{N\varphi_t^+}=(1-T_t(x)^2)^{\frac{N}{2}-1}T_t'(x).
\end{equation*}

By the change of variables $y=T_t(x)$, we obtain
\begin{equation*}
\int_{-1}^{1}(1-x^2)^{\frac{N}{2}-1}e^{N\varphi_t^+}dx=\int_{-1}^{1}(1-y^2)^{\frac{N-2}{2}}dy=\frac{\sqrt{\pi}\Gamma(\frac{N}{2})}{\Gamma(\frac{N+1}{2})}.
\end{equation*}

Similarly, we have
\begin{equation*}
\int_{-1}^{1}(1-x^2)^{\frac{N}{2}-1}e^{N\varphi_t^-}dx=\frac{\sqrt{\pi}\Gamma(\frac{N}{2})}{\Gamma(\frac{N+1}{2})}.
\end{equation*}

\textbf{Claim 1: $\varphi_t^\pm$ satisfy \eqref{axial} with $\alpha=1$.} 
Indeed, we have
\begin{equation*}
\frac{d^k}{dx^k}\varphi_t^+(x)=(k-1)!2^kt^k(1+t^2-2tx)^{-k}
\end{equation*}
and
\begin{equation*}
\frac{d^k}{dx^k}e^{N\varphi_t^+(x)}=(N)_k2^kt^k(1-t^2)^{N}(1+t^2-2tx)^{-N-k},
\end{equation*}
where 
\begin{equation*}
(s)_k:=
\begin{cases}
    s(s+1)\cdots (s+k-1),&\qquad k\geq 1,\\
    1,&\qquad k=0
\end{cases}    
\end{equation*}
is the Pochhammer symbol for $s\in \mathbb{R}$ and $k\in \mathbb{N}$.

Using \eqref{Rodrigues} and integrating by parts $k$ times, we get for $k\geq 1$,
\begin{equation*}
\int_{-1}^{1}(1-x^2)^{\frac{N}{2}-1}\varphi_t^+F_k dx=\kappa_{k,N}(k-1)!2^kt^k\int_{-1}^{+1}(1-x^2)^{\frac{N}{2}+k-1}(1+t^2-2tx)^{-k}dx
\end{equation*}
and
\begin{equation*}
\int_{-1}^{1}(1-x^2)^{\frac{N}{2}-1}e^{N\varphi_t^+}F_k dx=\kappa_{k,N}(N)_k2^kt^k(1-t^2)^{N}\int_{-1}^{+1}(1-x^2)^{\frac{N}{2}+k-1}(1+t^2-2tx)^{-N-k}dx
\end{equation*}
for some constant $\kappa_{k,N}>0$.

With the change of variables $y=T_t(x)$, we have
\begin{equation*}
\begin{aligned}
(1-t^2)^{N}\int_{-1}^{+1}(1-x^2)^{\frac{N}{2}+k-1}(1+t^2-2tx)^{-N-k}dx
=&\int_{-1}^{+1}(1-y^2)^{\frac{N}{2}+k-1}(1+t^2+2ty)^{-k}dy\\
=&\int_{-1}^{+1}(1-y^2)^{\frac{N}{2}+k-1}(1+t^2-2ty)^{-k}dy.
\end{aligned}
\end{equation*}

Then we have
\begin{equation*}
\frac{\Gamma(k+N)}{\Gamma(k)}\int_{-1}^{1}(1-x^2)^{\frac{N}{2}-1}\varphi_t^+F_k dx=(N-1)!\int_{-1}^{1}(1-x^2)^{\frac{N}{2}-1}e^{N\varphi_t^+}F_k dx.
\end{equation*}

By \eqref{Paneitzeigen}, it follows that
\begin{align*}
\int_{-1}^{1}(1-x^2)^{\frac{N}{2}-1}P_{r,N}\varphi_t^+F_k dx=&\int_{-1}^{1}(1-x^2)^{\frac{N}{2}-1}\varphi_t^+P_{r,N}F_k dx\\=&(N-1)!\int_{-1}^{1}(1-x^2)^{\frac{N}{2}-1}e^{N\varphi_t^+}F_k dx.
\end{align*}

Since
\begin{equation*}
\int_{-1}^{1}(1-x^2)^{\frac{N}{2}-1}F_k dx=0,
\end{equation*}
we have
\begin{equation*}
\int_{-1}^{1}(1-x^2)^{\frac{N}{2}-1}P_{r,N}\varphi_t^+F_k dx=(N-1)!\int_{-1}^{1}(1-x^2)^{\frac{N}{2}-1}(e^{N\varphi_t^+}-1)F_k dx
\end{equation*}
for $k\geq 1$. For $k=0$, both sides are zero and the equality holds automatically. Since the $\{F_k\}$ are complete in $L^2((-1,1);(1-x^2)^{\frac{N}{2}-1}dx)$ (see \cite[Theorem 3.1.5]{Szego1975}) and both sides are
smooth, we have $\varphi_t^+$ satisfies \eqref{axial} with $\alpha=1$.
The same argument also holds for $\varphi_t^-$. Hence the claim follows.

\hspace{1in}

With the change of variable $y=T_t(x)$, we can compute
\begin{equation*}
\int_{-1}^{1}(1-x^2)^{\frac{N}{2}-1}\varphi_t^{\pm}dx=\frac{\sqrt{\pi}\Gamma(\frac{N}{2})}{\Gamma(\frac{N+1}{2})}\log(1-t^2)+O(1)
\end{equation*}
and
\begin{equation*}
\int_{-1}^{1}(1-x^2)^{\frac{N}{2}-1}\varphi_t^{\pm}e^{N\varphi_t^\pm}dx=-\frac{\sqrt{\pi}\Gamma(\frac{N}{2})}{\Gamma(\frac{N+1}{2})}\log(1-t^2)+O(1).
\end{equation*}

\textbf{Claim 2: }
\begin{equation}\label{claim 2}
\int_{-1}^{1}(1-x^2)^{\frac{N}{2}-1}\varphi_t^{\pm}P_{r,N}\varphi_t^{\pm}dx=-2(N-1)!\frac{\sqrt{\pi}\Gamma(\frac{N}{2})}{\Gamma(\frac{N+1}{2})}\log(1-t^2)+O(1).
\end{equation}
Indeed, write $\varphi_t^+=\log(1-t^2)+v_t^+$ with $v_t^+=-\log(1+t^2-2tx)$. Since $\varphi_t^+$ satisfies \eqref{axial} with $\alpha=1$, testing it against $\varphi_t^+$, we obtain that
\begin{equation*}
\begin{aligned}
\int_{-1}^{1}(1-x^2)^{\frac{N}{2}-1}\varphi_t^{+}P_{r,N}\varphi_t^{+}dx
=&(N-1)!\int_{-1}^{1}(1-x^2)^{\frac{N}{2}-1}\varphi_t^{+}(e^{N\varphi_t^+}-1)dx\\
=&(N-1)!\int_{-1}^{1}(1-x^2)^{\frac{N}{2}-1}v_t^{+}(e^{N\varphi_t^+}-1)dx
\end{aligned}
\end{equation*}

Since $|v_t^+(x)|\leq C(1+|\log(1-x)|)$, we have
\begin{equation}\label{vt}
\int_{-1}^{1}(1-x^2)^{\frac{N}{2}-1}v_t^{+}dx=O(1).    
\end{equation}

Set $x=\frac{(1+t)^2-(1-t)^2r}{(1+t)^2+(1-t)^2r}$. Then we have
\begin{equation}\label{vtvarphit}
\begin{aligned}
\int_{-1}^{1}(1-x^2)^{\frac{N}{2}-1}v_t^{+}e^{N\varphi_t^+}dx
=&2^{N-1}\int_0^{\infty}\frac{r^{\frac{N}{2}-1}}{(1+r)^N}\left[-2\log(1-t)-\log(1+r)+\log(1+\frac{(1-t)^2}{(1+t)^2}r)\right]dr\\
=&-\frac{2\sqrt{\pi}\Gamma(\frac{N}{2})}{\Gamma(\frac{N+1}{2})}\log(1-t^2)+O(1).
\end{aligned}
\end{equation}

Combining \eqref{vt} and \eqref{vtvarphit} gives \eqref{claim 2} for $\varphi_t^+$. The proof for $\varphi_t^-$ is similar.

\hspace{1in}

\textbf{Claim 3: }
\begin{equation*}
\int_{-1}^{1}(1-x^2)^{\frac{N}{2}-1}\varphi_t^{+}P_{r,N}\varphi_t^{-}dx=O(1).
\end{equation*}
Indeed, testing the equation for $\varphi_t^-$ against $\varphi_t^+$, we have
\begin{equation*}
\begin{aligned}
\int_{-1}^{1}(1-x^2)^{\frac{N}{2}-1}\varphi_t^{+}P_{r,N}\varphi_t^{-}dx
=&(N-1)!\int_{-1}^{1}(1-x^2)^{\frac{N}{2}-1}\varphi_t^{+}(e^{N\varphi_t^-}-1)dx\\
=&(N-1)!\int_{-1}^{1}(1-x^2)^{\frac{N}{2}-1}v_t^{+}(e^{N\varphi_t^-}-1)dx.
\end{aligned}
\end{equation*}

For $x\leq 0$, since $v_t^+$ is uniformly bounded and
\begin{equation}\label{varphit}
\int_{-1}^{1} (1-x^2)^{\frac{N}{2}-1}e^{N\varphi_t^-}dx=\frac{\sqrt{\pi}\Gamma(\frac{N}{2})}{\Gamma(\frac{N+1}{2})},  
\end{equation}
we have
\begin{equation*}
\int_{-1}^{0} (1-x^2)^{\frac{N}{2}-1}|v_t^{+}|e^{N\varphi_t^-}dx\leq C.  
\end{equation*}

For $x\geq 0$, we have
\begin{equation*}
e^{\varphi_t^-}\leq 1-t^2.
\end{equation*}
Hence, 
\begin{equation*}
\int_{0}^{1} (1-x^2)^{\frac{N}{2}-1}|v_t^{+}|e^{N\varphi_t^-}dx\leq C(1-t^2)^N\int_{0}^{1} (1-x^2)^{\frac{N}{2}-1}|v_t^{+}|dx  \leq C(1-t^2)^N.
\end{equation*}

Combining with \eqref{vt}, Claim 3 follows.

\hspace{1in}

Finally, define 
\begin{equation*}
u_t(x)=\varphi_t^++\varphi_t^--2\log(1-t^2).
\end{equation*}
Then $u_t$ is an even function with
\begin{equation*}
\int_{-1}^{1}x(1-x^2)^{\frac{N}{2}-1}e^{Nu_t}dx=0.
\end{equation*}
Hence, $u_t\in \mathcal{L}_{r,N}$.

\textbf{Claim 4: }
\begin{equation*}
\log\left(\frac{\Gamma(\frac{N+1}{2})}{\sqrt{\pi}\Gamma(\frac{N}{2})}\int_{-1}^{1}(1-x^2)^{\frac{N-2}{2}} e^{Nu_t} dx\right)=-N\log(1-t^2)+O(1).
\end{equation*}

Indeed, by definition of $u_t$, we have
\begin{equation*}
4^{-N}(1-t^2)^{-N}\max\{e^{N\varphi_t^+},e^{N\varphi_t^-}\}\leq e^{Nu_t}\leq (1-t^2)^{-N}\max\{e^{N\varphi_t^+},e^{N\varphi_t^-}\}.
\end{equation*}
Hence, we have
\begin{equation*}
4^{-N-1}(1-t^2)^{-N}\left(e^{N\varphi_t^+}+e^{N\varphi_t^-}\right)\leq e^{Nu_t}\leq (1-t^2)^{-N}\left(e^{N\varphi_t^+}+e^{N\varphi_t^-}\right).
\end{equation*}

By \eqref{varphit}, we have
\begin{equation*}
c(1-t^2)^{-N}\leq \frac{\Gamma(\frac{N+1}{2})}{\sqrt{\pi}\Gamma(\frac{N}{2})}\int_{-1}^{1}(1-x^2)^\frac{N-2}{2} e^{Nu_t} dx\leq C(1-t^2)^{-N}
\end{equation*}
for some $c,C>0$. Then Claim 4 follows.

By Claim 2 and Claim 3, we have
\begin{equation*}
\begin{aligned}
\int_{-1}^{1}(1-x^2)^{\frac{N}{2}-1}u_tP_{r,N}u_tdx
=&\int_{-1}^{1}(1-x^2)^{\frac{N}{2}-1}\varphi_t^+P_{r,N}\varphi_t^+dx+\int_{-1}^{1}(1-x^2)^{\frac{N}{2}-1}\varphi_t^-P_{r,N}\varphi_t^-dx\\
&+2\int_{-1}^{1}(1-x^2)^{\frac{N}{2}-1}\varphi_t^+P_{r,N}\varphi_t^-dx\\
=&-4(N-1)!\frac{\sqrt{\pi}\Gamma(\frac{N}{2})}{\Gamma(\frac{N+1}{2})}\log(1-t^2)+O(1).
\end{aligned}
\end{equation*}

Meanwhile, since
\begin{equation*}
u_t=-\log(1+t^2-2tx)-\log(1+t^2+2tx),
\end{equation*}
we have
\begin{equation*}
\int_{-1}^{1}(1-x^2)^{\frac{N}{2}-1}u_tdx=O(1).
\end{equation*}

Combining with Claim 4, we obtain
\begin{equation*}
\mathcal{I}_{\alpha,N}(u_t)=-(N-1)!\frac{\sqrt{\pi}\Gamma(\frac{N}{2})}{\Gamma(\frac{N+1}{2})}(2\alpha-1)\log(1-t^2)+O(1).
\end{equation*}

If $\alpha<\frac{1}{2}$, then we have $\mathcal{I}_{\alpha,N}(u_t)\to -\infty$ as $t\to 1^-$, a contradiction. This proves the sharpness of $\alpha=\frac{1}{2}$.
\end{proof}

\section{Technicalities}\label{Technicalities}
In this appendix, we record the differentiability needed for the constraint argument.

\begin{lemma}\label{Differentiable}
The functionals
\begin{equation*}
Z(u):=\int_{\mathbb{S}^N}e^{Nu}dw,\qquad C(u):=\int_{\mathbb{S}^N}\xi e^{Nu}dw
\end{equation*}
are continuously Fr\'echet differentiable on subspace of $H^{\frac{N}{2}}(\mathbb{S}^N)$ with
\begin{equation*}
\delta Z(u)[h]=N\int_{\mathbb{S}^N}he^{Nu}dw,\qquad \delta C(u)[h]:=N\int_{\mathbb{S}^N}\xi he^{Nu}dw.
\end{equation*}
\end{lemma}
\begin{proof}
Note that for any real number $p,q$, we have
\begin{equation*}
|e^{p+q}-e^p-qe^p|\leq\frac12 q^2e^{|p|+|q|}.
\end{equation*}

With $p=Nu$ and $q=Nh$, H\"older's inequality gives
\begin{align*}
\int_{\mathbb{S}^N}|h|^2e^{N|u|+N|h|}dw
&\leq \|h\|_{L^6(\mathbb{S}^N)}^2
\left(\int_{\mathbb{S}^N}e^{3N|u|}dw\right)^{1/3}
\left(\int_{\mathbb{S}^N}e^{3N|h|}dw\right)^{1/3}.
\end{align*}

By Sobolev embedding $H^{N/2}(\mathbb{S}^N)\hookrightarrow L^6(\mathbb{S}^N)$ and Beckner's inequality with $\alpha=1$ applied to $\pm3u$ and $\pm3h$, we have the differentiability of $Z(u)$. To prove continuity of $\delta Z(u)$, note that if $u_j\to u$ in $H^{N/2}(\mathbb{S}^N)$, the inequality
\begin{equation*}
|e^{Nu_j}-e^{Nu}|
\leq N|u_j-u|e^{N|u|+N|u_j-u|}
\end{equation*}
along with the same H\"older's inequality and locally uniform exponential bounds imply 
\begin{equation*}
e^{Nu_j}\to e^{Nu} \text{ in } L^{\frac{6}{5}}(\mathbb{S}^N).    
\end{equation*}

The proof for the vector-valued functional $C(u)$ is identical componentwise, since each coordinate function $\xi_i$ is bounded.

\end{proof}

\section*{Acknowledgements}
 The research of  C. Gui is supported by NSFC Key Program (Grant No.12531010), University of Macau research grants CPG2024-00016-FST, CPG2025-00032-FST, CPG2026-00027-FST, SRG2023-00011-FST, MYRGGRG2023-00139-FST-UMDF, UMDF Professorial Fellowship
of Mathematics, Macao SAR FDCT 0003/2023/RIA1 and Macao SAR FDCT 0024/2023/RIB1. The research of J. Wei is partially supported by General Research Grant of HKSAR GRF 14309824. The authors acknowledge the use of AI tools. All mathematical arguments and proofs in the final manuscript were checked and written by the authors.


\medskip

\begin{thebibliography}{99}


\bibitem{Beckner1993}
W.~Beckner,
\emph{Sharp Sobolev inequalities on the sphere and the Moser--Trudinger inequality},
Ann. of Math. (2) \textbf{138} (1993), no.~1, 213--242.

\bibitem{CLY2019}
J.~S.~Case, Y.-J.~Lin, and W.~Yuan,
\emph{Conformally variational Riemannian invariants},
Trans. Amer. Math. Soc. \textbf{371} (2019), no.~11, 8217--8254.

\bibitem{CM2023}
J.~S.~Case and A.~Malchiodi,
\emph{A factorization of the GJMS operators of special Einstein products and applications},
J. Lond. Math. Soc. (2) \textbf{110} (2024), no.~5,
Paper No.~e70023, 17 pp.

\bibitem{ChangGui202}
S.-Y.~A.~Chang and C.~Gui,
\emph{A sharp inequality on the exponentiation of functions on the sphere},
Comm. Pure Appl. Math. \textbf{76} (2023), no.~6, 1303--1326.

\bibitem{ChangHang2022}
S.-Y.~A.~Chang and F.~B.~Hang,
\emph{Improved Moser--Trudinger--Onofri inequality under constraints},
Comm. Pure Appl. Math. \textbf{75} (2022), no.~1, 197--220.

\bibitem{ChangYang1987}
S.-Y.~A.~Chang and P.~C.~Yang,
\emph{Prescribing Gaussian curvature on $S^2$},
Acta Math. \textbf{159} (1987), nos.~3--4, 215--259.

\bibitem{ChangYang1988}
S.-Y.~A.~Chang and P.~C.~Yang,
\emph{Conformal deformation of metrics on $S^2$},
J. Differential Geom. \textbf{27} (1988), no.~2, 259--296.

\bibitem{ChangYang1995}
S.-Y.~A.~Chang and P.~C.~Yang,
\emph{Extremal metrics of zeta function determinants on $4$-manifolds},
Ann. of Math. (2) \textbf{142} (1995), no.~1, 171--212.

\bibitem{ChangYang1997}
S.-Y.~A.~Chang and P.~C.~Yang,
\emph{On uniqueness of solutions of $n$th-order differential equations in conformal geometry},
Math. Res. Lett. \textbf{4} (1997), no.~1, 91--102.

\bibitem{DHL2000}
Z.~Djadli, E.~Hebey, and M.~Ledoux,
\emph{Paneitz-type operators and applications},
Duke Math. J. \textbf{104} (2000), no.~1, 129--169.

\bibitem{DM2008}
Z.~Djadli and A.~Malchiodi,
\emph{Existence of conformal metrics with constant $Q$-curvature},
Ann. of Math. (2) \textbf{168} (2008), no.~3, 813--858.


\bibitem{FFGG1998}
J.~Feldman, R.~Froese, N.~Ghoussoub, and C.~Gui,
\emph{An improved Moser--Aubin--Onofri inequality for axially symmetric functions on $S^2$},
Calc. Var. Partial Differential Equations \textbf{6} (1998), no.~2, 95--104.

\bibitem{FG2013}
C.~Fefferman and C.~R.~Graham,
\emph{Juhl's formulae for GJMS operators and $Q$-curvatures},
J. Amer. Math. Soc. \textbf{26} (2013), no.~4, 1191--1207.

\bibitem{FrankLieb2012}
R.~L. Frank and E.~H. Lieb,
\emph{A new, rearrangement-free proof of the sharp Hardy--Littlewood--Sobolev inequality},
in \emph{Spectral Theory, Function Spaces and Inequalities},
Oper. Theory Adv. Appl., Vol.~219,
Birkh\"auser/Springer Basel AG, Basel, 2012, pp.~55--67.

\bibitem{GL2010}
N.~Ghoussoub and C.-S.~Lin,
\emph{On the best constant in the Moser--Onofri--Aubin inequality},
Comm. Math. Phys. \textbf{298} (2010), no.~3, 869--878.

\bibitem{GR2007}
I.S.~Gradshteyn, I.M.~Ryzhik, \emph{Table of Integrals, Series and Products,} 7th ed., Academic Press, San Diego, 2007.

\bibitem{GJMS1992}
C.~R.~Graham, R.~Jenne, L.~J.~Mason, and G.~A.~J.~Sparling,
\emph{Conformally invariant powers of the Laplacian. I. Existence},
J. Lond. Math. Soc. (2) \textbf{46} (1992), no.~3, 557--565.

\bibitem{GM2018}
C.~Gui and A.~Moradifam,
\emph{The sphere covering inequality and its applications},
Invent. Math. \textbf{214} (2018), no.~3, 1169--1204.

\bibitem{GHM2020}
C.~Gui, F.~Hang, and A.~Moradifam,
\emph{The sphere covering inequality and its dual},
Comm. Pure Appl. Math. \textbf{73} (2020), no.~12, 2685--2707.

\bibitem{GHX2021}
C.~Gui, Y.~Hu, and W.~Xie,
\emph{Improved Beckner's inequality for axially symmetric functions on $\mathbb{S}^4$},
Rev. Mat. Iberoam. \textbf{40} (2024), no.~1, 355--388.

\bibitem{GHW2022}
C.~Gui, Y.~Hu, and W.~Xie,
\emph{Improved Beckner's inequality for axially symmetric functions on $\mathbb{S}^n$},
J. Funct. Anal. \textbf{282} (2022), no.~5,
Paper No.~109335, 47 pp.

\bibitem{GLWY2025}
C.~Gui, T.~Li, J.~Wei, and Z.~Ye,
\emph{Sharp Beckner's inequalities for axially symmetric functions on $\mathbb{S}^6$ and $\mathbb{S}^8$},
Adv. Math. \textbf{480} (2025), Paper No.~110487, 64 pp.


\bibitem{GW2000}
C.~Gui and J.~Wei,
\emph{On a sharp Moser--Aubin--Onofri inequality for functions on $S^2$ with symmetry},
Pacific J. Math. \textbf{194} (2000), no.~2, 349--358.

\bibitem{GurMal2015}
M.~J.~Gursky and A.~Malchiodi,
\emph{A strong maximum principle for the Paneitz operator and a non-local flow for the $Q$-curvature},
J. Eur. Math. Soc. (JEMS) \textbf{17} (2015), no.~9, 2137--2173.

\bibitem{JinLiXiong2017}
T.~Jin, Y.~Y.~Li, and J.~Xiong,
\emph{The Nirenberg problem and its generalizations: a unified approach},
Math. Ann. \textbf{369} (2017), nos.~1--2, 109--151.

\bibitem{LWY2022}
T.~Li, J.~Wei, and Z.~Ye,
\emph{On sharp Beckner's inequality for axially symmetric functions on $\mathbb{S}^4$},
Math. Res. Lett. \textbf{31} (2024), no.~5, 1523--1550.

\bibitem{LiXiong2019}
Y.~Y.~Li and J.~Xiong,
\emph{Compactness of conformal metrics with constant $Q$-curvature. I},
Adv. Math. \textbf{345} (2019), 116--160.

\bibitem{Mal2006}
A.~Malchiodi,
\emph{Compactness of solutions to some geometric fourth-order equations},
J. Reine Angew. Math. \textbf{594} (2006), 137--174.

\bibitem{Mori1998}
M.~Morimoto,
\emph{Analytic Functionals on the Sphere},Translations of Mathematical Monographs, Vol.~178,
American Mathematical Society, Providence, RI, 1998.



\bibitem{Paneitz2008}
S.~M.~Paneitz,
\emph{A Quartic Conformally Covariant Differential Operator for Arbitrary
    Pseudo-Riemannian Manifolds (Summary)},
SIGMA Symmetry Integrability Geom. Methods Appl. \textbf{4} (2008),
Paper No.~036, 3 pp.

\bibitem{Seeley1967}
R.~T. Seeley,
\emph{Complex powers of an elliptic operator},
in \emph{Singular Integrals (Proc. Sympos. Pure Math., Chicago, Ill., 1966)},
Proc. Sympos. Pure Math., Vol.~10, Amer. Math. Soc., Providence, RI, 1967, pp.~288--307.

\bibitem{SSTW2019}
Y.~Shi, J.~Sun, G.~Tian, and D.~Wei,
\emph{Uniqueness of the mean field equation and rigidity of Hawking mass},
Calc. Var. Partial Differential Equations \textbf{58} (2019), no.~2,
Paper No.~41, 16 pp.

 \bibitem{Szego1975}
 G.~Szeg\H{o},
 \emph{Orthogonal Polynomials},
 4th ed., American Mathematical Society Colloquium Publications,
 Vol.~23, American Mathematical Society, Providence, RI, 1975.

\bibitem{Taylor1991}
M.~E. Taylor,
\emph{Pseudodifferential Operators and Nonlinear PDE},
Progress in Mathematics, Vol.~100, Birkh\"auser Boston, Boston, MA, 1991.

\bibitem{Widom1988}
H.~Widom,
\emph{On an inequality of Osgood, Phillips and Sarnak},
Proc. Amer. Math. Soc. \textbf{102} (1988), no.~3, 773--774.

\bibitem{WX1998}
J.~Wei and X.~Xu,
\emph{On conformal deformations of metrics on $S^n$},
J. Funct. Anal. \textbf{157} (1998), no.~1, 292--325.

\bibitem{Zhang2025}
S. Zhang, \emph{The moving plane method and the uniqueness of high-order elliptic equation with GJMS operator}, Proc. Roy. Soc. Edinburgh Sect. A, First View (2025), 1–45.

\end{thebibliography}
\end{document}